\documentclass[12pt,twoside]{amsart}
\usepackage{amsmath}
\usepackage{amsthm}
\usepackage{amsfonts}
\usepackage{amssymb}
\usepackage{latexsym}
\usepackage{mathrsfs}
\usepackage{amsmath}
\usepackage{amsthm}
\usepackage{amsfonts}
\usepackage{amssymb}
\usepackage{latexsym}
\usepackage{geometry}
\usepackage{dsfont}
\usepackage[dvips]{graphicx}
\usepackage{color}
\usepackage[all]{xy}

\date{}
\allowdisplaybreaks[4] \footskip=15pt
\renewcommand{\uppercasenonmath}[1]{}

\usepackage{graphicx,amssymb}
\usepackage[all]{xy}
\usepackage{amsmath}
\usepackage{enumerate}
\allowdisplaybreaks
\usepackage{amsthm}
\usepackage{color}
\usepackage{amsmath}
\theoremstyle{plain}
\newtheorem{theorem}{Theorem}[section]
\newtheorem{proposition}[theorem]{Proposition}
\newtheorem{lemma}[theorem]{Lemma}
\newtheorem{corollary}[theorem]{Corollary}

\newtheorem*{open question}{Open Question}
\newtheorem{definition}[theorem]{Definition}
\newcommand{\C}{\boldsymbol{{C}}}
\newcommand{\K}{\boldsymbol{{K}}}
\newcommand{\A}{\boldsymbol{{A}}}
\newcommand{\D}{\boldsymbol{{D}}}
\theoremstyle{definition}

\theoremstyle{remark}
\newtheorem{remark}[theorem]{Remark}

\newcommand{\Proj}{\mathcal{P}}

\newcommand{\Add}{\mathrm{Add}}

\newcommand{\Id}{\mathrm{Id}}

\def\fPD{{\rm fPD}}
\def\fpD{{\rm fpD}}

\def\Mod{{\rm Mod}}
\def\fPD{{\rm fPD}}

\def\cwd{{\rm w.gl.dim}}

\def\Hom{{\rm Hom}}
\def\Ext{{\rm Ext}}

\def\Ker{{\rm Ker}}

\def\Im{{\rm Im}}
\def\Coker{{\rm Coker}}

\def\id{{\rm id}}

\newcommand{\FPR}{\mathcal{FPR}}
\def\fP{\mathcal{P}}
\def\fI{\mathcal{I}}
\def\Cone{{\rm Cone}}
\def\Hh{{\rm H}}

\begin{document}
	\begin{center}
		{\large  \bf Little finitistic dimensions and generalized derived categories\footnote{
$^*$ Corresponding author.}}
		
		\vspace{0.5cm}   Xiaolei Zhang$^{a}$,~ Tiwei Zhao$^{b,*}$,~ Dingguo Wang$^{b}$

		{\footnotesize  a.\ School of Mathematics and Statistics, Shandong University of Technology, Zibo 255049, P.R. China\\
			b.\ School of Mathematical Sciences, Qufu Normal University, Qufu 273165, P.R. China
			
			E-mail: zxlrghj@163.com, tiweizhao@qfnu.edu.cn, dgwang@qfnu.edu.cn\\
		}
	\end{center}
	
	\bigskip
	\centerline { \bf  Abstract}
	\bigskip
	\leftskip10truemm \rightskip10truemm \noindent
	
	In this paper, we introduced a generalization of the derived category, which is called the $n$-derived category and denoted by $\D_{n}(R)$, of a given ring $R$ for each $n\in\mathbb{N}\cup\{\infty\}$. The $n$-derived category of a ring is proved to be very closely connected with its left little finitistic dimension. We also introduce and investigate the notions of $n$-exact sequences, $n$-projective (resp., $n$-injective) modules and $n$-exact complexes. In particular, we characterize the left little finitistic dimensions in terms of all above notions.  Finally, we  build a connection of the classical derived categories and $n$-derived categories.
	\vbox to 0.3cm{}\\
	
	\bigskip

	\leftskip0truemm \rightskip0truemm
	\bigskip
	
	\section{Introduction}
	
	In this paper, $R$  is always an associative ring with identity and $R$-Mod is the category of all left $R$-modules. All $R$-modules are left $R$-modules unless otherwise stated. Given a ring $R$, the left little finitistic dimension $l.\fpD(R)$ of $R$, defined to be the supremum of the projective dimensions of  finitely generated $R$-modules with finite projective dimension,  was first introduced by Bass \cite{B05}.  The left little finitistic dimensions of left  Noetherian rings, especially left  Artinian rings, have attracted many algebraists (e.g.  \cite{Chen-Xi17,Wei09,Xi06,Zheng-Huang20,Huisgen92}). However, it is difficult to dig deeper for non-Noetherian rings, since the syzygies of finitely generated modules are  not finitely generated over non-Noetherian rings in general. To amend this gap, Glaz \cite{g89} revised the notion of left little finitistic dimension of a ring $R$.  Let $R$ be a ring and  $M$  a left $R$-module. Then $M$ is said to have a finite projective resolution, denoted by $M\in\FPR$, if  there exists an integer $n$ and an exact sequence
	$$0\rightarrow P_n\rightarrow P_{n-1}\rightarrow \dots\rightarrow P_1\rightarrow P_0\rightarrow M\rightarrow 0$$
	with each $P_i$  finitely generated projective. We denote $\Proj^{\leq n}$ (resp., $\Proj^{<n}$) to be the class of left $R$-modules with projective dimensions at most (resp., less than) $n$ in  $\FPR$ for each $n=\{0,1,2,\cdots\}\cup\{\infty\}$. The left little finitistic dimension of a ring $R$, introduced by Glaz \cite{g89} and denoted by $\fPD(R)$, is defined to be the supremum of the projective dimensions of $R$-modules in $\FPR$. Clearly, $l.\fPD(R)\leq n$ if and only if $\FPR=\Proj^{\leq n}$. Some characterizations of Glaz version's little finitistic dimension  of commutative rings are given by the first author and Wang \cite{ZW21}.
	
	Let $(\mathscr{A},\mathscr{E})$ be an exact category in the sense of Quillen \cite{Q73} and $\K(\mathscr{A})$ its homotopy category. Let $\A(\mathscr{A})$ be the full subcategory of $\K(\mathscr{A})$  consisting complexes $\cdots\rightarrow X^{i}\xrightarrow{d^{i}} X^{i+1}\rightarrow{}\cdots$ where each $d^{i}$ can naturally decompose into an admissible monomorphism and  an admissible epimorphism. Then $\A(\mathscr{A})$ is a thick subcategory of $\K(\mathscr{A})$.  By Neeman \cite{N90}, one can define the derived category of the exact category  $(\mathscr{A},\mathscr{E})$  to be Verdier quotient of $\K(\mathscr{A})$ by $\A(\mathscr{A})$, i.e.,  $\D(\mathscr{A}):=\K(\mathscr{A})/\A(\mathscr{A})$. For example, if the exact structure consists of all short exact sequences in $R$-Mod, then the Neeman version's derived  category is exactly the classical derived category $\D(R)$. In 2016, Zheng and Huang \cite{ZZ16} utilized the  exact structure of all pure  short exact sequences in $R$-Mod to investigate the notion of pure derived categories. Recently, Tan, Wang and Zhao generalized the pure derived categories to $n$-derived categories  in terms of $n$-exact sequences for each $n=\{0,1,2,\cdots\}\cup\{\infty\}$. They also built a connection of the classical derived categories and $n$-derived categories.
	
	To establish a relationship between the little finitistic dimension of a ring $R$  and the relative  derived categories introduced by  Neeman \cite{N90},  we consider,  for each $n=\{0,1,2,\cdots\}\cup\{\infty\}$, the exact structure consisting of all $n$-exact sequences, i.e.,  short exact sequences in $R$-Mod which keep exactness under the functor $\Hom_R(P,-)$ where $P\in \Proj^{< n+1}.$ The  corresponding derived categories, denoted by $\D_{n}(R)$, is called the $n$-derived category of $R$. Moreover, the classical bounded derived category $\D^b(R)$ can be seen as a Verdier quotient of bounded $n$-derived category $\D^b_{n}(R)$ (see Theorem \ref{last}). To obtain these results, we also introduce and study the notions of $n$-projective (resp., $n$-injective) modules and $n$-exact complexes. The left little finitistic dimension can be  characterized in terms of all these notions (See Theorem \ref{fpd-exact-seq}, Corollary \ref{fnfn+1}, Corollary \ref{fpd-exact-comp} and Theorem \ref{fpd-derived}). Besides, we also compare our $n$-exact notions with Zheng and Huang's pure exact notions defined in \cite{ZZ16}.
	
	The paper is organized as follows. In Section 2, we introduce the class $\mathscr{E}_{n}$ of  $n$-exact sequences in terms of exactness of the Hom-functor, and give the Cohen' theorem for $n$-exact sequences if the basic ring is commutative. We also show that, for a given  ring $R$,  $l.\fPD(R)\leq n$ if and only if  $\mathscr{E}_{n}=\mathscr{E}_{{n+1}}$, if and only if $\mathscr{E}_{n}=\mathscr{E}_{\infty}$.
	To keep the exactness of Hom-functor on $n$-exact sequences, we give the notions of $n$-projective modules and $n$-injective modules in Section 3. We also show that, for any ring (resp., commutative ring) $R$, $l.\fPD(R)\leq n$ if and only if every $(n+1)$-projective (resp., $(n+1)$-injective) module is  ${n}$-projective (resp., ${n}$-injective), if and only if every ${\infty}$-projective (resp., ${\infty}$-injective) module is  ${n}$-projective (resp., ${n}$-injective). Besides, we also study the (pre)cover and (pre)envelope properties of $n$-projective modules and $n$-injective modules. To get of derived  level of $n$-exactness, we introduce the $n$-exact complex in Section 4. We also characterize the little finitistic dimension by $n$-exact complexes which is similar with the case of $n$-exact sequences. In Section 5, we finally introduce the $n$-derived category $\D_{n}(R)$ of a ring $R$.  For any $m> n$, $\D^{\ast}_{n}(R)$ is a  Verdier quotient of $ \D^{\ast}_{m}(R)$ by some thick subcategory. Moreover,   $l.\fPD(R)\leq n$ if and only if   $\D^{\ast}_{n}(R)$ is naturally triangulated equivalence to $ \D^{\ast}_{m}(R)$ for any (or some) $\D^{\ast}_{m}(R)$ with $m> n$. Besides, some classical results can also be extended to  $n$-derived categories.  In Section 6, we connect the  classical bounded derived category $\D^b(R)$ with our bounded $n$-derived category $\D^b_{n}(R)$. In fact, we show that there is a triangulated equivalence $\D^b(R)\cong \D^b_{n}(R)/\K^b_{ac}(\fP_n)$, where $\K^b_{ac}(\fP_n)$ is the homotopy category of bounded exact complexes of $n$-projective modules.

	\section{Exact sequences induced by $\FPR$}
	
	For a ring $R$, we denote by $R$-mod to be the full subcategory of all left $R$-modules that have finitely generated projective resolutions, i.e., those left $R$-modules $M$  fitting into an exact sequence
	$$\dots\rightarrow P_m\rightarrow P_{m-1}\rightarrow \dots\rightarrow P_1\rightarrow P_0\rightarrow M\rightarrow 0$$ where each $P_i$ is finitely generated projective.
	
	\begin{definition} {\rm  Let $M$ be a left $R$-module. We say $M$ that has a finite projective resolution $($of length at most $n)$, if there is an exact sequence
			$$0\rightarrow P_n\rightarrow P_{n-1}\rightarrow \dots\rightarrow P_1\rightarrow P_0\rightarrow M\rightarrow 0$$ where each $P_i$ is a finitely generated projective left $R$-module.}
	\end{definition}
	
We denote by $\Proj^{< \infty}$ the class of all $R$-modules with finite projective resolutions, $\Proj^{< 0}=\{0\}$ and $\Proj^{< n}$ the subclass of $R$-modules with  finite projective resolution of length less than $n$ for some positive integer $n$.
	
	Throughout this paper,  we always fix $n$ to be a non-negative integer or $\infty$. Besides, we always identity $\infty+1$ and $\infty$.

	The left little finitistic dimension of a ring $R$, denoted by $l.\fPD(R)$, is defined to be the supremum of the projective dimensions of all left $R$-modules in $\FPR$. Clearly, $l.\fPD(R)\leq n$ if and only if $\Proj^{<\infty}=\Proj^{< n+1}$.

	\begin{definition} {\rm  A short exact sequence of left $R$-modules $$0\rightarrow A\rightarrow B\rightarrow C\rightarrow 0$$ is said to be {\it $n$-exact} provided that $$0\rightarrow \Hom_R(M,A)\rightarrow \Hom_R(M,B)\rightarrow \Hom_R(M,C)\rightarrow 0$$ is exact for any $M\in \Proj^{<n+1}$.}
	\end{definition}

	Trivially, short $0$-exact sequences are precisely short exact sequences. We denote by  $\mathscr{E}_{n}$ the class of all $n$-exact sequences. Then, for each $n$,  $\mathscr{E}_{n}$ is an  exact structure of $R$-Mod in the sense of Quillen \cite{Q73}.
	We have the following inclusions:
	$$\mathscr{E}_{\infty}\subseteq\cdots\subseteq\mathscr{E}_{{n+1}}\subseteq \mathscr{E}_{n} \subseteq\cdots\subseteq \mathscr{E}_{{1}}\subseteq \mathscr{E}_{{0}}.$$

	The following result can be seen as   Cohen' theorem for  $n$-exact sequences (see \cite[Theorem 3.69]{R09}).
	
	\begin{theorem}\label{cohen} Let $R$ be a ring and  $\Upsilon: 0\rightarrow A\xrightarrow{f} B\xrightarrow{g} C\rightarrow 0$ be  a short exact sequence of $R$-modules. Then the following statements are equivalent for each $n$.
		\begin{enumerate}
			\item  $\Upsilon: 0\rightarrow A\xrightarrow{f} B\xrightarrow{g} C\rightarrow 0$  is $n$-exact.
			\item For any commutative diagram with  exact rows
			$$\xymatrix@R=25pt@C=25pt{
				0 \ar[r]^{}  & K\ar[d]_{u}\ar[r]^{h} &R^l \ar@{.>}[ld]_{w}\ar[d]^{v} & &  \\
				0 \ar[r]^{} & A\ar[r]^{f} & B \ar[r]^{g} &C\ar[r]^{} &  0\\}$$
			with $0<l<\infty$ and $ K\in \Proj^{<n}$, there exists $w:R^l\rightarrow A$ such that $u=wh$.
			\item For all $0<m<\infty$, $0<l<\infty$ and all systems of $R$-linear equations $(\mathscr{S})$ in the variables $x_i (1\leq i\leq l)$ with $a_j\in A (1\leq j\leq m)$, $r_{ij}(1\leq i\leq l,1\leq j\leq m)$
			$$(\mathscr{S})\ \ \sum\limits_{1\leq i\leq l}r_{ij}x_i=f(a_j) $$
			satisfying that $\langle \sum\limits_{1\leq i\leq l} r_{ij} e_i \mid 1\leq  j\leq m \rangle \in \Proj^{<n}$ where $\{e_i\mid 1\leq i\leq l\}$ is the basis of a free $R$-module, the following holds:
			\begin{center}
				$(\mathscr{S})$ has a solution in $f(A)$ whenever $(\mathscr{S})$ has a solution in $B$.
			\end{center}
		\end{enumerate}
Now suppose $R$ is a commutative ring. Then all above are equivalent to:

  $\textit{(4)}$ The natural sequence $0\rightarrow M\otimes_RA\xrightarrow{1\otimes f} M\otimes_RB\xrightarrow{1\otimes g} M\otimes_RC\rightarrow 0$ is exact for any $R$-module  $M\in \Proj^{<n+1}$.

	\end{theorem}
	
	\begin{proof} $(1)\Leftrightarrow (2)$  Consider the commutative diagram:
		$$\xymatrix@R=25pt@C=30pt{
			0 \ar[r]^{}& K\ar[d]_{u}\ar[r]^{h} &R^l\ar@{.>}[ld]_{w} \ar[d]^{v} \ar[r] &\Coker(h) \ar@{.>}[ld]_{t}\ar[d]^{s} \ar[r]  &0 \\
			0 \ar[r]^{} & A\ar[r]^{f} & B \ar[r]^{g}          &C\ar[r]^{}                 &  0\\}$$
Note that  $ K\in \Proj^{<n}$ if and only if  $\Coker(h)\in \Proj^{<n+1}$. It follows  from \cite[Exercise 1.60]{fk16} that there is an $R$-homomorphism $t: \Coker(h)\rightarrow B$ such that $s=gt$ if and only if  there is an $R$-homomorphism $w: R^l\rightarrow A$ such that $u=wh$.

		$(2)\Rightarrow (3)$ Set $\{x_1,\cdots,x_l\}$ to be a basis of $R^l$. Let     $(\mathscr{S})\ \ \sum\limits_{1\leq i\leq l}r_{ij}x_i=f(a_j)\ (1\leq j\leq m)$ be a family of equations such that  $\{z_j:=\sum\limits_{1\leq i\leq l}r_{ij}x_i\mid 1\leq j\leq m\}$ satisfies that $K:=\langle z_1,\cdots, z_m\rangle\in \Proj^{<n}$. Suppose  $(\mathscr{S})$ has  solutions $b_i\in B\ (1\leq i\leq l)$. Set $h:K\to R^l$ to be the natural embedding map, $v$ satisfies $v(x_i)=b_i$ and $u$ satisfies  $u(z_j)=a_j$ for all $i,j$. Then $f(u(z_j))=f(a_j)=\sum\limits_{1\leq i\leq l}r_{ij}b_i=v(h(z_j))$. Hence there is an $R$-homomorphism  $w:R^l\rightarrow A$ such that $u=wh$. Then $\sum\limits_{1\leq i\leq l}r_{ij}f(w(x_i))=f(u(\sum\limits_{1\leq i\leq l}r_{ij}x_i))=f(u(z_j))=f(a_j)\ (1\leq j\leq m)$. That is, $(\mathscr{S})$ has a solution in $f(A)$.

	$(3)\Rightarrow (2)$ Set $\{e_1,\cdots,e_l\}$ to be the natural basis of $R^l$,  $K:=\langle \sum\limits_{1\leq i\leq l}r_{ij}e_i  \mid 1\leq  j\leq m\rangle\in \Proj^{<n}$, $v(e_i)=b_i$ and $u(\sum\limits_{1\leq i\leq l}r_{ij}e_i)=a_j$ for each $i,j$. Then $\sum\limits_{1\leq i\leq l}r_{ij}b_i= \sum\limits_{1\leq i\leq l}r_{ij}v(e_i)=vh(\sum\limits_{1\leq i\leq l}r_{ij}e_i)=fu(\sum\limits_{1\leq i\leq l}r_{ij}e_i)=f(a_j)$. Hence there are $a'_i\in A$ such that $\sum\limits_{1\leq i\leq l}r_{ij}f(a'_i)=f(a_j)$. Setting $w:R^l\rightarrow A$ satisfying $w(e_i)=a'_i$ for each $i$. Then $wh(\sum\limits_{1\leq i\leq l}r_{ij}e_i )=\sum\limits_{1\leq i\leq l}r_{ij}w(e_i)=  \sum\limits_{1\leq i\leq l}r_{ij}a'_i=a_j=u(\sum\limits_{1\leq i\leq l}r_{ij}e_i )$. Hence $u=wh$.

Now, suppose $R$ is a commutative ring.
		
		$(3)\Rightarrow (4)$ We claim $1\otimes f:M\otimes_RA\rightarrow M\otimes_RB$ is a monomorphism for any  $M\in \Proj^{<n+1}$. Indeed, since $M\in \Proj^{<n+1}$,  there is an exact sequence $0\rightarrow K\rightarrow R^l\rightarrow M\rightarrow 0$ with $K=\langle \sum\limits_{1\leq i\leq l}r_{ij}e_i\mid  1\leq j\leq m\rangle\in \Proj^{<n}$. So $\sum\limits_{1\leq i\leq l}r_{ij}m_i=0$. Suppose $\sum\limits_{1\leq i\leq l}m_i\otimes a_i\in \Ker(1\otimes f)$ where $\{m_1,\cdots, m_l\}$ is a generating set of $M$. Then $\sum\limits_{1\leq i\leq l}m_i\otimes f(a_i)=0$  in $M\otimes_R B$. By \cite[Lemma 3.68]{R09}, there exists $b_j\in B$ such that $f(a_i)=\sum\limits_{1\leq j\leq m}r_{ij}b_j$ for all $i$. Therefore, there are $a'_j\in A$ such that $f(a_i)=\sum\limits_{1\leq j\leq m}r_{ij}f(a'_j)$ by assumption. Hence $\sum\limits_{1\leq i\leq l}m_i\otimes f(a_i)=\sum\limits_{1\leq i\leq l}m_i\otimes (\sum\limits_{1\leq j\leq m}r_{ij}f(a'_j))=\sum\limits_{1\leq j\leq m}(\sum\limits_{1\leq i\leq l}r_{ij}m_i)\otimes f(a'_j)=0$ in $M\otimes_R f(A)$. Since $f$ is a monomorphism, $\sum\limits_{1\leq i\leq l}m_i\otimes a_i=0$ in $M\otimes_R A$ and hence $1\otimes f$ is a monomorphism.
		
		$(4)\Rightarrow (3)$ Let $(\mathscr{S})\ \ \sum\limits_{1\leq i\leq l}r_{ij}x_i=f(a_j) $ be a systems of $R$-linear equations satisfying that $K:=\langle \sum\limits_{1\leq i\leq l}r_{ij}e_i  \mid 1\leq  j\leq m\rangle\in \Proj^{<n}$. Suppose $(\mathscr{S})$ has a solution $\{b_i\mid 1\leq i\leq l\}\subseteq B$. Set $R^l$ to be a finitely generated free $R$-module with the natural basis $\{e_1,\cdots,e_l\}$. So $M:=R^l/K\in \Proj^{<n+1}$.  Note that $$1\otimes f(\sum\limits_{1\leq i\leq l}\overline{e_i}\otimes a_i)=\sum\limits_{1\leq i\leq l}\overline{e_i}\otimes f(a_i)=\sum\limits_{1\leq i\leq l}\overline{e_i}\otimes (\sum\limits_{1\leq j\leq m}r_{ij}b_i)= \sum\limits_{1\leq j\leq m}(\sum\limits_{1\leq i\leq l}r_{ij}\overline{e_i})\otimes b_i=0.$$
		Since $1\otimes f$ is a monomorphism by $(4)$, we have $\sum\limits_{1\leq i\leq l}\overline{e_i}\otimes a_i=0$. Hence by  \cite[Lemma 3.68]{R09}, there exist elements $a'_i\in A$ such that $f(a_j)=\sum\limits_{1\leq i\leq l}r_{ij}f(a'_i)$. That is, $(\mathscr{S})$ has a solution in $f(A)$.

	\end{proof}
\begin{remark}  We must note that the equivalence $(1)\Leftrightarrow (4)$ in Theorem \ref{cohen} need not be true for non-commutative rings. Let $n\in\mathbb{N}$ and  $R$  a ring with weak global dimension $< n$. It follows that \cite[Corollary 2.9,  Corollary 2.11]{BG14} that a ring $R$ is right coherent if and only if [for all short exact sequences $\Upsilon: 0 \rightarrow A \rightarrow B\rightarrow C\rightarrow 0$ of left $R$-modules, $\Upsilon$ is pure exact iff $K \otimes_R \Upsilon$ is exact for all right $R$-modules $K\in\Proj^{<n}$], and $R$ is left coherent if and only if  [for all short exact sequences $\Upsilon: 0 \rightarrow A \rightarrow B\rightarrow C\rightarrow 0$ of left $R$-modules, $\Upsilon$ is pure exact iff $\Hom_R(T, \Upsilon)$ is exact for all left $R$-modules $T \in\Proj^{<n}$].

Let $S=\prod\limits_{\aleph_0}\mathbb{Z}_2$ be the product of $\aleph_0$ copies of the field $\mathbb{Z}_2$, $I=\bigoplus\limits_{\aleph_0}\mathbb{Z}_2$ be an ideal of $S$ and $T=S/I$ be the factor ring. Set $R=\begin{pmatrix}T&T\\ 0&S\end{pmatrix}$be the formal triangular matrix ring. Then $R$ be a left coherent ring with weak global dimension $1$ which is not right coherent (see \cite[Proposition 3.1]{C61} and \cite[Example 2.1.8]{ZC19}). Then there exists a non-pure short exact sequence $\Upsilon: 0 \rightarrow A \rightarrow B\rightarrow C\rightarrow 0$ of left $R$-modules such that $\Hom_R(T, \Upsilon)$ is exact for all left $R$-modules $T \in\Proj^{<2}$ as  $R$ is not right coherent. But $K \otimes_R \Upsilon$ is exact for all right $R$-modules $K\in\Proj^{<2}$ since $R$ is left coherent with weak global dimension $1$.
\end{remark}

We can characterize left little finitistic dimensions in terms of $n$-exact sequences.
	\begin{theorem}\label{fpd-exact-seq} Let $R$ be a ring. Then the following statements are equivalent.
		\begin{enumerate}
			\item $l.\fPD(R)\leq n$.
			\item $\mathscr{E}_{n}=\mathscr{E}_{{n+1}}$.
			\item $\mathscr{E}_{n}=\mathscr{E}_{{m}}$ for some $m>n$.
			\item $\mathscr{E}_{n}=\mathscr{E}_{{m}}$ for any $m>n$.
			\item $\mathscr{E}_{n}=\mathscr{E}_{\infty}$.
		\end{enumerate}
	\end{theorem}
	\begin{proof} $(1)\Rightarrow (5)$ Suppose  $l.\fPD(R)\leq n$. Then every $R$-module with a finite projective resolution has length at most $n$, that is, $\Proj^{<\infty}=\Proj^{< n+1}$. So $\mathscr{E}_{n}=\mathscr{E}_{\infty}$.
		
		$(5)\Rightarrow(4)\Rightarrow(3)\Rightarrow (2)$ Obvious.
		
		$(2)\Rightarrow (1)$ Suppose that $M$ is an $R$-module in $\Proj^{< \infty}$. We can assume  $M\in \Proj^{< m+1}$ for some $m>n$, i.e., $M$ can fit into an exact sequence
		$$0\rightarrow P_{m}\xrightarrow{d_{m}} P_{m-1}\xrightarrow{d_{m-1}} \dots\rightarrow P_{n+1}\xrightarrow{d_{n+1}} P_{n}\xrightarrow{d_{n}} \dots\rightarrow P_1\xrightarrow{d_1} P_0\xrightarrow{d_0} M\rightarrow 0,$$ where each $P_i$ is finitely generated projective. We can assume that $m=kn+s$ with $0<s\leq n$. Consider the  exact sequence $$0\rightarrow P_{m}\rightarrow  \dots\rightarrow P_{kn+1}\xrightarrow{d_{kn+1}} P_{kn}\xrightarrow{d_{kn}}P_{kn-1}\rightarrow  \dots \rightarrow P_{(k-1)n}\xrightarrow{d_{(k-1)n}}\Im(d_{(k-1)n})\rightarrow 0.$$
		Claim that the following short exact sequence $$0\rightarrow \Im(d_{kn+1})\rightarrow P_{kn}\rightarrow\Im(d_{kn}) \rightarrow 0\ \ \ \ \ \ \ \ \ \ (\Upsilon_k)$$
		is in $\mathscr{E}_{n}$. Indeed, let $L$ be an $R$-module in $\Proj^{< n+1}$. Then the projective dimension of $L$ is at most $n$. Thus $\Ext^1_R(L,\Im(d_{kn}))\cong \Ext^{n+1}_R(L,\Im(d_{(k-1)n}))=0$ by dimension shift.  So the natural sequence $0\rightarrow \Hom_R(L,\Im(d_{kn+1}))\rightarrow \Hom_R(L,P_{kn})\rightarrow \Hom_R(L,\Im(d_{kn}))\rightarrow 0$ is exact. So the exact sequence $\Upsilon_k$ is in $\mathscr{E}_{n}$. Hence $\Upsilon_k$ is in $\mathscr{E}_{{n+1}}$ by assumption. Since $\Im(d_{kn})\in \Proj^{< n+2}$, the exact sequence $\Upsilon_k$ splits, and thus $\Im(d_{kn})$ is a finitely generated projective  $R$-module, and so $M\in \Proj^{< kn+1}$. Iterating these steps, we have  $M\in \Proj^{< n+1}$. Consequently $l.\fPD(R)\leq n$.
	\end{proof}

	\section{Projectives and injectives induced by $\FPR$}

	Let $R$ be a ring. Recall that an $R$-module $M$ is said to be  projective if $M$ is projective with respect to all  exact sequences. The notion of injective modules is defined dually.

	\begin{definition}{\rm  Let $R$ be a ring. An $R$-module $M$ is said to be {\it $n$-projective}
			if $M$ is projective with respect to all $n$-exact sequences, i.e.,  for any $n$-exact sequence $ 0\rightarrow A\rightarrow B\rightarrow C\rightarrow 0$, the natural exact sequence $$0\rightarrow \Hom_R(M, A)\rightarrow \Hom_R(M,B)\rightarrow \Hom_R(M,C)\rightarrow 0$$ is exact. Dually, we can define the notion of {\it $n$-injective} modules.}
	\end{definition}

	Obviously, the class of  $n$-projective modules is closed under direct sums and  direct summands, and the class of  $n$-injective modules is closed under direct products and  direct summands for each $n$.
	Since a exact sequence is precisely a  $0$-exact sequence, we have  $0$-projective modules are exactly  projective modules, and $0$-injective modules are exactly  injective modules. We always denote by $\fP_n$ (resp., $\fI_n$) the class of all $n$-projective (resp., $n$-injective) $R$-modules. Trivially,
	
	$$\fP_0\subseteq  \fP_1\subseteq \cdots\subseteq \fP_n\subseteq \cdots\subseteq \fP_\infty,$$
	$$\fI_0\subseteq  \fI_1\subseteq \cdots\subseteq \fI_n\subseteq \cdots\subseteq \fI_\infty.$$

	\begin{lemma}\label{char-pure-inj} Let $R$ be a commutative ring and $M$  an $R$-module in $\Proj^{<n+1}$. Then the $R$-module $M^{+}:=\Hom_{\mathbb{Z}}(M,\mathbb{Q}/\mathbb{Z})$ is $n$-injective.
	\end{lemma}
	\begin{proof} Let $0\rightarrow A\rightarrow B\rightarrow C\rightarrow 0$ be an $n$-exact sequence. Then, by Theorem \ref{cohen}, we have  the following natural exact sequence $$0\rightarrow M\otimes_RA\rightarrow M\otimes_RB\rightarrow M\otimes_RC\rightarrow 0.$$ So the sequence  $$0\rightarrow \Hom_R(C,M^{+})\rightarrow \Hom_R(B,M^{+})\rightarrow \Hom_R(A,M^{+})\rightarrow 0$$
		is exact which implies that  $M^{+}$ is $n$-injective.
	\end{proof}

	Let  $ \mathcal{F}$ be a  class of $R$-modules and $M$ an $R$-module. An $R$-homomorphism $f: F\rightarrow M$ with  $F\in \mathcal{F}$ is said to be an $\mathcal{F}$-precover if the natural homomorphism $\Hom_R(F',F)\rightarrow \Hom_R(F',M)$ is an epimorphism for any $F'\in \mathcal{F}$. That is, for any $R$-homomorphism $g:F'\rightarrow M$, there exists $h:F'\rightarrow F$ such that the following diagram commutates:
	$$\xymatrix{
		&&F' \ar[rd]^g  \ar@{.>}[ld]_h  \\
		&F \ar[rr]^f &&M.}$$
	If, moreover, any homomorphism $h$ such that $f=f\circ h$  is an isomorphism,  then $f$ is said to be an $\mathcal{F}$-cover.
	
	Dually, we can define $\mathcal{F}$-preenvelopes and $\mathcal{F}$-envelopes.
	
	Note that $\mathcal{F}$-(pre)covers and $\mathcal{F}$-(pre)envelopes do not exist in general. If any $R$-module has an $\mathcal{F}$-(pre)cover (resp., $\mathcal{F}$-(pre)envelope), then we say $\mathcal{F}$ is  (pre)enveloping (resp., (pre)covering). If $\mathcal{F}$-covers or $\mathcal{F}$-envelopes of an $R$-module exist, then it is unique up to isomorphism.
	
	Next, we will consider the $\fI_n$-(pre)envelopes  and $\fP_n$-(pre)covers  of $R$-modules. Note that every $\fI_n$-preenvelope  is a monomorphism and  every $\fP_n$-precover is an epimorphism. Now, let $\Upsilon: 0\rightarrow A\xrightarrow{f} B\xrightarrow{g} C\rightarrow 0$ be an exact sequence. Then
	
	(1) $g$ is a $\fP_n$-precovers   if and only if $B$ is $n$-projective and $\Upsilon$ is $n$-exact;

	(2) $f$ is an $\fI_n$-preenvelopes   if and only if $B$ is $n$-injective  and $\Upsilon$ is $n$-exact.

	\begin{theorem}\label{pre-c-e-fpn} The following statements hold.
		\begin{enumerate}
			\item  $\fP_n$ is precovering.
			\item Let $R$ be a commutative ring. Then $\fI_n$ is preenveloping.
		\end{enumerate}
	\end{theorem}
	\begin{proof} (1) We denote by $\mathcal{S}_n$ the family of all representative elements of isomorphism classes of all $R$-modules in $\Proj^{<n+1}$. Then $\mathcal{S}_n$ is a set. Let $M$ be an $R$-module and $P\in \mathcal{S}_n$. Consider $R$-homomorphism $f_P:P^{(\Hom_R(P,M))}\rightarrow M$ defined by $(p)\mapsto \sum f(p)$. Then any $R$-homomorphism $P\rightarrow M$ factor through $f_P$. Considering the $R$-homomorphism
		$$f:=\bigoplus\limits_{P\in \mathcal{S}_n}f_P\in \Hom_R(\bigoplus\limits_{P\in \mathcal{S}_n}P^{(\Hom_R(P,M))},M),$$
		we have that every $R$-homomorphism $g:P\rightarrow M$ with $P\in\Proj^{<n+1}$ can factor through $f,$ and so $f$ is $n$-exact. Hence there is an exact sequence $$\Hom_R(P,\bigoplus\limits_{P\in \mathcal{S}_n}P^{(\Hom_R(P,M))})\rightarrow \Hom_R(P,M) \rightarrow 0.$$ Setting $P'=R$, we have that $f$ is also an epimorphism. Since  $\bigoplus\limits_{P\in \mathcal{S}_n}P^{(\Hom_R(P,M))}\in \fP_n$, $f$ is a $\fP_n$-precover of $M$.
		
		(2) Let $M$ be an $R$-module. It follows from the proof of $(1)$ that there is an $n$-exact sequence $0\rightarrow K\rightarrow \bigoplus P_i\rightarrow M^{+}\rightarrow 0$, where each $P_i$ is an $R$-module in $\Proj^{<n+1}$. Then the sequence $0\rightarrow M^{++}\rightarrow  (\bigoplus P_i)^{+}\rightarrow K^{+}\rightarrow 0$ is exact. Let $P\in \Proj^{<n+1}$.  Then the sequence  $0\rightarrow P\otimes_RK\rightarrow P\otimes_R\bigoplus P_i\rightarrow P\otimes_RM^{+}\rightarrow 0$ is exact, and so is $0\rightarrow (P\otimes_RM^{+}) ^{+} \rightarrow (P\otimes_R\bigoplus P_i)^{+}\rightarrow (P\otimes_RK) ^{+}\rightarrow 0$, that is, the sequence $0\rightarrow\Hom_R(P, M^{++})\rightarrow \Hom_R(P, \bigoplus P_i^{+})\rightarrow \Hom_R(P, K^{+})\rightarrow 0$ is exact. Hence $0\rightarrow M^{++}\rightarrow (\bigoplus P_i)^{+}\rightarrow K^{+}\rightarrow 0$ is $n$-exact. On the other hand, the pure exact sequence $0\rightarrow M\rightarrow M^{++}\rightarrow M^{++}/M\rightarrow 0$ is obviously $n$-exact. Note that  $(\bigoplus P_i)^{+}$ is $n$-injective by Lemma \ref{char-pure-inj}. Consider the following push-out diagram:
		$$\xymatrix@R=14pt@C=14pt{ &  0\ar[d]&0\ar[d]&&\\
			&  M\ar[d]\ar@{=}[r]&M\ar[d]^{g}&&\\
			0 \ar[r]^{} & M^{++}\ar[d]\ar[r]^{} & (\bigoplus P_i)^{+}\ar[d]\ar[r]^{} &K^{+}\ar@{=}[d]\ar[r]^{} &  0\\
			0 \ar[r]^{} & M^{++}/M\ar[d]\ar[r]^{} & X \ar[d]\ar[r]^{} &K^{+}\ar[r]^{} &  0\\
			& 0 &0 &&.\\}$$
		It is easy to check that the middle column is $n$-exact. Hence $g:M\rightarrow (\bigoplus P_i)^{+}$ is an $\fI_n$-preenvelope of $M$.
	\end{proof}

	\begin{proposition}\label{str-fpr} The following statements hold.
		\begin{enumerate}
			\item A left $R$-module $P$ is  $n$-projective if and only if it is a direct summand of  direct sum of $R$-modules in  $\Proj^{<n+1}$ .
			\item Let $R$ be a commutative ring. Then  an  $R$-module $I$ is  $n$-injective if and only if it is a direct summand of $\prod\limits_{i\in \Lambda} P_i^{+}$, where each  $P_i$ is an $R$-module in $\Proj^{<n+1}$.
		\end{enumerate}
	\end{proposition}
	\begin{proof} We only prove $(1)$ since $(2)$ can be proved similarly. Let $P$ be an left $R$-module.
		Suppose that $P$ is  $n$-projective. Then there is an $n$-exact sequence $$(\Psi)\ \ \ \ 0\rightarrow K\rightarrow F\rightarrow P\rightarrow 0$$  with $F$ a direct sum of $R$-modules in $\Proj^{<n+1}$ by the proof of Theorem \ref{pre-c-e-fpn}(1). Since $P$ is $n$-projective, $(\Psi)$ splits. The converse follows from that  the class of $n$-projective modules contains $\Proj^{<n+1}$ and is closed under direct summands and direct sums.
	\end{proof}

	\begin{remark}\label{3.5} We give some remarks on $n$-projective modules and $n$-injective modules.
		\begin{enumerate}
			
			\item Trivially, an $R$-module is  $0$-projective if and only if it is projective; and an $R$-module is  $0$-injective if and only if it is injective.
			\item Obviously, every $n$-projective $R$-module has projective dimensions at most $n$. However, the converse does not hold in general for $n\geq 1$. For example, let $R=\mathbb{Z}$, then its quotient field $\mathbb{Q}$ has  projective dimension equal to $1$. But  $\mathbb{Q}$ is not $1$-projective. On the contrary, suppose that $\mathbb{Q}$ is a direct summand of $\bigoplus\limits_{i\in \Lambda}M_i$, where $M_i$ is finitely generated $\mathbb{Z}$-module. Note that every finitely generated $\mathbb{Z}$-module can be decomposed as a finite direct sum of finitely generated free modules and finite torsion modules. Since $\mathbb{Q}$ is  divisible and torsion-free, $\Hom_{\mathbb{Z}}(\mathbb{Q},T)=0$ for any torsion module $T$. So  we can assume that each $M_i$ is  finitely generated free, which is a contradiction as $\mathbb{Q}$ is not projective.
			
			\item Similarly, every $n$-injective has injective dimensions at most $n$. However, the converse does not hold in general for $n\geq 1$. For example, we consider the example in $(2)$. Although $\id_{\mathbb{Z}}\mathbb{Z}=1$, $\mathbb{Z}$ is not $1$-injective as a  $\mathbb{Z}$-module. Indeed, suppose $\mathbb{Z}$ is a direct summand of $\prod\limits_{i\in \Lambda} P_i^{+}$ with $P_i\in \Proj^{<2}$. Then $\mathbb{Z}$ is a direct summand of some $(\prod\limits_{i\in \Lambda} F_i^{+})\bigoplus \prod\limits_{i\in \Lambda} T_i^{+}$ where $F_i$ is a finitely generated free module and $T_i$ is a non-zero finite  torsion  $\mathbb{Z}$-module for each $i$. Since $\prod\limits_{i\in \Lambda} F_i^{+}$ is an injective  $\mathbb{Z}$-module and $\mathbb{Z}$ has no non-zero injective submodules, we can assume $\mathbb{Z}$ is a direct summand of $\prod\limits_{i\in \Lambda} T_i^{+}$ which is isomorphic to $\prod\limits_{i\in \Lambda} T_i$ by
			\cite[Exercise 7.1(1)]{F15}. However, this would not happen. On the contrary, suppose this happens.  Note that each finite torsion  $\mathbb{Z}$-module  $T_i$ is a finite direct sum (also a direct product) of indecomposable $\mathbb{Z}$-modules of the form  $\mathbb{Z}/p_i^{n_i}\mathbb{Z}$ where $p_i$ is a prime and $n_i$ is a positive integer.  Then $\mathbb{Z}$ is isomorphic to a direct summand of some $\prod\limits\mathbb{Z}/p_i^{n_i}\mathbb{Z}$. Note that each $\mathbb{Z}/p_i^{n_i}\mathbb{Z}$ is indecomposable pure-injective by \cite[Example 2.35(1)]{gt}. So the indecomposable $\mathbb{Z}$-module $\mathbb{Z}$ itself is also pure-injective, which is a contradiction (see \cite[Example 2.35(1)]{gt} again).

			
		\end{enumerate}
	\end{remark}

	\begin{proposition}\label{exact-pp} Let $\Upsilon: 0\rightarrow A\xrightarrow{f} B\xrightarrow{g} C\rightarrow 0$ be  a short exact sequence of left $R$-modules. Then the following statements are equivalent.
		\begin{enumerate}
			\item  $\Upsilon: 0\rightarrow A\xrightarrow{f} B\xrightarrow{g} C\rightarrow 0$  is $n$-exact.
			\item  The induced sequence $$0\rightarrow \Hom_R(P,A)\rightarrow \Hom_R(P,B)\rightarrow \Hom_R(P,C)\rightarrow 0$$ is exact for any $n$-projective module $P$.

		\end{enumerate}
Moreover, if $R$ is a commutative ring, then both above are equivalent to:

  $\textit{(3)}$ The induced sequence $$0\rightarrow \Hom_R(C,I)\rightarrow \Hom_R(B,I)\rightarrow \Hom_R(A,I)\rightarrow 0$$ is exact for any $n$-injective module $I$.

	\end{proposition}
	\begin{proof} It can be easily deduced by Proposition \ref{str-fpr}.
	\end{proof}

	Given a ring $R$, we denote by $\cwd(R)$ the weak global dimension of $R$, i.e., the supremum of flat dimensions of all $R$-modules.
	\begin{proposition}\label{fpn fp} Let $R$ be a ring,  and $\mathscr{E}_{fp}$  the class of all pure exact sequences. Then  $\mathscr{E}_{fp}=\mathscr{E}_{n}$ if and only if   $R$ is a left coherent ring  over which every finitely presented $R$-module has  projective dimension $<n+1$.  Moreover, if $n$ is a non-negative integer, then $\mathscr{E}_{fp}=\mathscr{E}_{n}$ if and only if   $R$ is a left coherent ring with $\cwd(R)<n+1$.
	\end{proposition}
	
	\begin{proof} Let $R$ be a left coherent ring over which every finitely presented $R$-module has  projective dimension $<n+1$.  Suppose that $M$ is a finitely presented $R$-module. Then $M$ can fit into an exact sequence $\cdots\rightarrow P_{m}\rightarrow \cdots\rightarrow P_1\rightarrow P_0\rightarrow M\rightarrow 0,$ where each $P_i$ is finitely generated projective. Since $M$ has  projective dimension $<n+1$, wed have $M\in \Proj^{< n+1}$. Hence  $\mathscr{E}_{fp}=\mathscr{E}_{n}$.
		
		
		On the other hand, suppose $\mathscr{E}_{fp}=\mathscr{E}_{n}$. Then every pure projective $R$-module is $n$-projective. By Proposition \ref{str-fpr}, every finitely presented $R$-module $M$ is a direct summand of  direct sum of $R$-modules in  $\Proj^{<n+1}$.  Actually,  $M$ is a direct summand of finite direct sum of $R$-modules in  $\Proj^{<n+1}$.  So  $M\in \Proj^{<n+1}$, and hence $M$ has  projective dimension $<n+1$. Note that every finitely presented left $R$-module $M$ is also super finitely presented. Therefore, $R$ is left coherent by \cite[Proposition 2.1]{BG14}.
		
		The last part follows by \cite[Theorem 3.3]{S70} that a left coherent ring $R$ has weak global dimension $<n+1$ if and only if the projective dimension of every finitely presented $R$-modules $<n+1$.
	\end{proof}

	\begin{corollary}\label{fnfn+1} Let $R$ be a ring. Then the following statements hold.
		\begin{enumerate}
			
			\item  The following statements are equivalent:
			
			\begin{enumerate} \item  $l.\fPD(R)\leq n$;
				\item  every $(n+1)$-projective left $R$-module is ${n}$-projective;
				\item any   ${m}$-projective   left $R$-module is ${n}$-projective for some $m>n$;
				\item any   ${m}$-projective left $R$-module is ${n}$-projective for any $m>n$;
				\item any ${\infty}$-projective left $R$-module is ${n}$-projective.
			\end{enumerate}
			\item The following statements are equivalent:
			\begin{enumerate} \item
				$R$ is a left coherent ring over which every finitely presented $R$-module has  projective dimension $<n+1$;
				\item  every pure projective left $R$-module is ${n}$-projective.
			\end{enumerate}
	
		\end{enumerate}
	\end{corollary}
	\begin{proof} The statement (1) follows by Theorem \ref{fpd-exact-seq} and Proposition \ref{exact-pp}. The statement (2) follows by \cite[Lemma 2.1.23]{P09}, Propositions  \ref{exact-pp} and  \ref{fpn fp}.
	\end{proof}
	
Note that if $R$ is a commutative ring, then Corollary \ref{fnfn+1} holds for the ``injective'' version.
	
	Let $\mathscr{M}$ be a class of homomorphisms of $R$-modules and $\mathscr{F}$ be a class of $R$-modules. Recall from \cite{EJ11} that a pair $(\mathscr{M},\mathscr{F})$ is called an injective structure on $R$-modules provided that
	\begin{enumerate}
		\item  $F\in \mathscr{F}$ if and only if $\Hom_R(N,F)\rightarrow \Hom_R(N,F)\rightarrow 0$ is exact for any $M\rightarrow N\in \mathscr{M}$
		\item $M\rightarrow N\in \mathscr{M}$ if and only if $\Hom_R(N,F)\rightarrow \Hom_R(N,F)\rightarrow 0$ is exact for any $F\in \mathscr{F}$.
		\item $\mathscr{F}$ is preenveloping.
	\end{enumerate}
	
	Let $\mathscr{P}$ be a class of right $R$-modules. We say that $(\mathscr{M},\mathscr{F})$ is determined by $\mathscr{P}$ provided that the following condition holds:
	\begin{center}
		$M\rightarrow N\in \mathscr{M}$ if and only if $0\rightarrow G\otimes_RM\rightarrow G\otimes_RN$ is exact for any $G\in \mathscr{P}$.
	\end{center}
	It follows from \cite[Theorem 6.6.4]{EJ11} that if an injective structure $(\mathscr{M},\mathscr{F})$ is determined by a class $\mathscr{P}$, then $\mathscr{F}$ is enveloping.
	
	\begin{theorem} Let $R$ be a commutative ring. Then every $R$-module has an $\fI_n$-envelope.
	\end{theorem}
	\begin{proof} Denote by $\mathscr{M}_n$  the class of all $n$-monomorphisms of $R$-modules. Then $(\mathscr{M}_n, \fP_n)$  is an  injective structure  determined by $\Proj^{< n+1}$ by Theorem \ref{cohen} and Proposition \ref{exact-pp}. Hence every $R$-module has an $\fI_n$-envelope by \cite[Theorem 6.6.4]{EJ11}.
	\end{proof}

	Recall that a ring $R$ is called a left perfect ring provided that every direct limit of left projective $R$-modules is projective, that is, for any direct system $\{P_i\mid i\in \Lambda\}$ of left $R$-modules in $\Add(R)$, the canonical  epimorphism $\bigoplus\limits_{i\in \Lambda}P_i\rightarrow {\lim\limits_{\longrightarrow}}_{i\in \Lambda}P_i\rightarrow 0$ is split. To extend the perfect properties of rings to that of  modules, the
	authors in \cite{AS06}  called  an $R$-module $M$  has  perfect decompositions  provided that  for any direct system $\{M_i\mid i\in \Lambda\}$ of left $R$-modules in $\Add(M)$, the canonical  epimorphism $\bigoplus\limits_{i\in \Lambda}M_i\rightarrow {\lim\limits_{\longrightarrow}}_{i\in \Lambda}M_i\rightarrow 0$ is split. Obviously, a ring $R$ is a left perfect ring if and only if $R$ has perfect decompositions as a left $R$-module.  One of the importance of  perfect decompositions comes from \cite[Corollary 2.3.]{S21} which states that  an $R$-module $M$ which is direct sum of $\aleph_n$-generated modules for some $n<\omega$ has a perfect decomposition, if and only if  $\Add(M)$ is closed under direct limits if and only if  $\Add(M)$ is  a covering class, if and only if  each module from the class  $\lim\limits_{\longrightarrow}\Add(M)$ has an  $\Add(M)$-cover.
	Consider the $R$-module $\bigoplus\limits_{P\in [\Proj^{<n+1}]}P$, where $[\Proj^{<n+1}]$ denotes the representatives of isomorphism class of modules in  $\Proj^{<n+1}$. Since each $R$-module in $\Proj^{<n+1}$ is finitely generated, then by \cite[Corollary 2.3.]{S21} we have the following result.

	\begin{theorem} Let $R$ be a ring. Then the following statements are equivalent.
		\begin{enumerate}
			\item The $R$-module $\bigoplus\limits_{P\in [\Proj^{<n+1}]}P$  has  perfect decompositions.
			\item  Every direct limit of $n$-projective $R$-modules is
			$n$-projective.
			\item Every $R$-module has an $n$-projective cover.
		\end{enumerate}
	\end{theorem}
	
	\begin{remark}\label{ppr-M}
		It follows by the above  that every $R$-module admits  $n$-projective resolutions and  $n$-injective coresolutions, that is, for each $M \in R$-\Mod,
		there exists an exact sequence
		$$\cdots \rightarrow P_1 \rightarrow P_0 \rightarrow M \rightarrow 0$$
		with each $P_i\in \fP_n $ and it remains exact after applying $\Hom_R(P, -)$ for any $P\in\fP_n,$
		and there exists an exact sequence
		$$0\rightarrow M\rightarrow E_0\rightarrow E_1 \rightarrow\cdots $$
		with each $E_i \in \fI_n $ and it remains exact after applying $\Hom_R(-, E)$ for any $E\in \fI_n.$
	\end{remark}

	\section{The relative  exact complexes induced by $\FPR$}
	
	We denote by $\C(R)$ and $\K(R)$ the category of complexes of $R$-Mod and homotopy category of $R$-Mod, respectively. For any $X^{\bullet}\in \C(R)$, we write
	$$X^{\bullet}:=\cdots\rightarrow X^{i-1}\xrightarrow{d_{X^{\bullet}}^{i-1}}X^{i}\xrightarrow{d_{X^{\bullet}}^{i}} X^{i+1}\xrightarrow{d_{X^{\bullet}}^{i+1}} X^{i+2}\rightarrow\cdots$$
	Let $M$ be an $R$-module, we can regard $M$ as a stalk complex, that is,  a complex concentrated in degree $0$.

	Let $X^{\bullet}\in \C(R)$. If $X^{i}=0$ for $i\gg 0$, then $X^{\bullet}$ is said to be bounded above. If $X^{i}=0$ for $i\ll 0$, then $X^{\bullet}$ is said to be bounded below. Moreover, $X^{\bullet}$ is said to be bounded provided that it is both bounded above and  bounded below.
	A cochain map $f:X^{\bullet}\rightarrow Y^{\bullet}$ is said to be a quasi-isomorphism if it induces the isomorphic homology groups, and $f$ is said to be homotopy equivalent provided that there exists a cochain map $g:Y^{\bullet}\rightarrow X^{\bullet}$ such that there exist homotopies $gf\thicksim \Id_{X^{\bullet}}$ and $fg\thicksim \Id_{Y^{\bullet}}$.

	\begin{definition}{\rm
			A complex $X^{\bullet}$ is said to be $n$-exact at $i$  if  the sequence $0\rightarrow \Ker(d_{X^{\bullet}}^{i})\rightarrow X^{i}\rightarrow \Coker(d_{X^{\bullet}}^{i-1})\rightarrow 0$ is $n$-exact.  A complex $X^{\bullet}$ is said to be $n$-exact if it is $n$-exact at $i$ for any   $i\in \mathbb{Z}$.}
	\end{definition}
	
	\begin{proposition}\label{Char-fn-exact complex} Let  $X^{\bullet}\in \C(R)$. Then the following statements are equivalent.
		\begin{enumerate}
			\item  $X^{\bullet}$  is an $n$-exact complex.
			\item  The complex $\Hom_R(P,X^{\bullet})$ is exact for any $P\in \Proj^{<n+1}$.
			\item  The complex $\Hom_R(P,X^{\bullet})$ is exact for any $n$-projective module $P$.
		\end{enumerate}
	\end{proposition}
	\begin{proof} It follows by Corollary \ref{exact-pp}.
	\end{proof}
	
	Since $R$ itself is $n$-projective, every $n$-exact complex is exact.  Since  $\Hom_R(P,-)$ commutes with direct limits for any $P\in \Proj^{<n+1}$, every direct limit of  $n$-exact complexes is $n$-exact. Trivially,  ${0}$-exact complexes are exactly  exact complexes. Every  ${m}$-exact complex  is ${n}$-exact for any $m> n$. Next we characterize when any ${n}$-exact complex is ${m}$-exact with $m> n$.
	
	\begin{corollary}\label{fpd-exact-comp} Let $R$ be a ring. Then the following statements are equivalent.
		\begin{enumerate}
			\item $l.\fPD(R)\leq n$.
			\item Any $($resp., bounded, bounded above, bounded below$)$ ${n}$-exact complex is $(n+1)$-exact.
			\item Any  $($resp., bounded, bounded above, bounded below$)$ ${n}$-exact complex is ${m}$-exact for some $m> n$.
			\item Any  $($resp., bounded, bounded above, bounded below$)$ ${n}$-exact complex is ${m}$-exact for any $m> n$.
			\item Any  $($resp., bounded, bounded above, bounded below$)$ ${n}$-exact complex is ${\infty}$-exact.
		\end{enumerate}
	\end{corollary}
	\begin{proof} It can easily be deduced by Theorem \ref{fpd-exact-seq}, Corollary \ref{fnfn+1} and Proposition \ref{Char-fn-exact complex}.
	\end{proof}

	Recall from \cite[Definition 2.3]{ZZ16} that a complex $X^{\bullet}$ is said to be pure exact  if  the sequence $0\rightarrow \Ker(d_{X^{\bullet}}^{i})\rightarrow X^{i}\rightarrow \Coker(d_{X^{\bullet}}^{i-1})\rightarrow 0$ is pure exact for any $i\in\mathbb{Z}$. Trivially, any pure exact complex is ${n}$-exact. On the other hand, the following result holds.
	
	\begin{corollary}\label{fpn fp-exacomp} Let $R$ be a ring.   Any  ${n}$-exact complex is pure exact if and only if   $R$ is a left coherent ring  over which every finitely presented $R$-module has  projective dimension $<n+1$.  Moreover, if $n$ is a non-negative integer, then it is also equivalent to that   $R$ is a left coherent ring with $\cwd(R)<n+1$.
	\end{corollary}
	\begin{proof} It can easily be deduced by Proposition \ref{fpn fp} , Corollary \ref{fnfn+1} and Proposition \ref{Char-fn-exact complex}.
	\end{proof}

	Next, we recall the notion of the  mapping cone of a cochain map in $\C(R)$. Let $f:X^{\bullet}\rightarrow Y^{\bullet}$ be a cochain map in $\C(R)$. Denote by $\Cone^{\bullet}(f)$ the mapping cone of $f$, which is defined by
	
	\begin{center}
		$\Cone^{\bullet}(f)^{i}=X^{i+1}\oplus X^i$ with\
		$d_{\Cone(f)^{\bullet}}^{i}=
		\begin{pmatrix}
			-d_{X^{\bullet}}^{i+1}&0\\
			f^{i+1}&d_{Y^{\bullet}}^{i}
		\end{pmatrix}$ for all $i\in\mathbb{Z}$.
	\end{center}
	It is well-known that  a cochain map $f:X^{\bullet}\rightarrow Y^{\bullet}$ is a quasi-isomorphism if and only if $\Cone^{\bullet}(f)$ is exact, and $f$ is a homotopy equivalence if and only if $\Cone^{\bullet}(f)$ is contractible.
	
	\begin{definition} {\rm Let  $f:X^{\bullet}\rightarrow Y^{\bullet}$ be a cochain map of complexes. We say that $f$ is an  {\it $n$-quasi-isomorphism} provided that its mapping cone $\Cone^{\bullet}(f)$ is an $n$-exact complex.}
	\end{definition}
	
	To continue, we recall the notions of  $\Hom$-complexes and Tensor-complexes. Let $X^{\bullet}, Y^{\bullet}\in \C(R)$. We use $\Hom_R(X^{\bullet}, Y^{\bullet})$ to denote the $\Hom$-complex, that is, a complex of $\mathbb{Z}$-modules
	$$\cdots\rightarrow \prod\limits_{i\in\mathbb{Z}}\Hom_R(X^i,Y^{m+i})\xrightarrow{d^m}
	\prod\limits_{i\in\mathbb{Z}}\Hom_R(X^i,Y^{m+1+i})\rightarrow\cdots$$
	where $ \prod\limits_{i\in\mathbb{Z}}\Hom_R(X^i,Y^{m+i})$ lies in degree $m$. For any $\phi\in \prod\limits_{i\in\mathbb{Z}}\Hom_R(X^i,Y^{m+i})$,  the $m$-th differential $d^m$ satisfies that $d^m(\phi)=(d^{i+m}_{Y^{\bullet}}\phi^i-(-1)^m\phi^{i+1}d^i_{X^{\bullet}})_{i\in\mathbb{Z}}.$
	It is well-known that for any $m\in\mathbb{Z}$, we have $\Hh^m\Hom_R(X^{\bullet}, Y^{\bullet})=\Hom_{\K(R)}(X^{\bullet}, Y^{\bullet}[m])$.

	\begin{proposition}\label{Char-fn-quas-iso}  Let  $f:X^{\bullet}\rightarrow Y^{\bullet}$ be a cochain map of complexes. Then the following statements are equivalent.
		\begin{enumerate}
			\item  $f$  is an  $n$-quasi-isomorphism.
			\item  $\Hom_R(P,f):\Hom_R(P,X^{\bullet})\rightarrow \Hom_R(P,Y^{\bullet})$ is a quasi-isomorphism for any $P\in \Proj^{<n+1}$.
			\item  $\Hom_R(P,f):\Hom_R(P,X^{\bullet})\rightarrow \Hom_R(P,Y^{\bullet})$ is a quasi-isomorphism for any $n$-projective module $P$.
		\end{enumerate}
	\end{proposition}
	\begin{proof}  It follows by Proposition \ref{Char-fn-exact complex}.
	\end{proof}

	\begin{corollary}\label{cor-fn-p-qiso}  Let
		$$\xymatrix@R=25pt@C=30pt{
			X_1^{\bullet}\ar[d]^{h_1}\ar[r]^{}&X_2^{\bullet}\ar[r]^{}\ar[d]^{h_2}
			&X_3^{\bullet}\ar[r]^{}\ar[d]^{h_3}&X_1^{\bullet}[1] \ar[d]^{h_{1}[1]}\\
			Y_1^{\bullet}\ar[r]^{}&Y_2^{\bullet}\ar[r]^{}&Y_3^{\bullet}\ar[r]^{}&Y_1^{\bullet}[1]\\}$$
		be a morphism of triangles in $\K(R)$. If $h_1$ and $h_2$ are  $n$-quasi-isomorphisms, then $h_3$ is also an $n$-quasi-isomorphism.
	\end{corollary}
	\begin{proof} Let $P\in \Proj^{<n+1}$. Then, by using The five lemma and  Proposition \ref{Char-fn-quas-iso},   the assertion follows from the following commutative diagram
		$$\xymatrix@R=25pt@C=30pt{
			\Hom_R(P,X_1^{\bullet})\ar[d]^{\Hom_R(P,h_1)}\ar[r]^{}&\Hom_R(P,X_2^{\bullet})\ar[r]^{}\ar[d]^{\Hom_R(P,h_2)}
			&\Hom_R(P,X_3^{\bullet})\ar[r]^{}\ar[d]^{\Hom_R(P,h_3)}&\Hom_R(P,X_1^{\bullet}[1] ) \ar[d]^{\Hom_R(P,h_{1}[1])}\\
			\Hom_R(P,Y_1^{\bullet})\ar[r]^{}&\Hom_R(P,Y_2^{\bullet})\ar[r]^{}&\Hom_R(P,Y_3^{\bullet})\ar[r]^{}&\Hom_R(P,Y_1^{\bullet}[1])\\}$$
	\end{proof}

	We denote by $\K^{-}(\fP_n)$  the bounded above  homotopy category of $\fP_n$, i.e., the full subcategories of  $\K^{-}(R)$  with each term in $\fP_n$. Then we have the following two results.

	\begin{lemma} \label{char-pex-pquiso} Let $f:X^{\bullet}\rightarrow Y^{\bullet}$ be a cochain map of complexes. Then the following two statements hold.
		\begin{enumerate}
			\item The following statements are equivalent.
			\begin{enumerate}
				\item  $X^{\bullet}$  is an $n$-exact complex.
				\item  The complex $\Hom_R(P^{\bullet},X^{\bullet})$ is exact for any  $P^{\bullet}$ in $\K^{-}(\fP_n)$.
			\end{enumerate}

			\item  The following statements are equivalent.
			\begin{enumerate}[(a)]
				\item  $f$  is an  $n$-quasi-isomorphism.
				\item  $\Hom_R(P^{\bullet},f):\Hom_R(P^{\bullet},X^{\bullet})\rightarrow \Hom_R(P^{\bullet},Y^{\bullet})$ is a quasi-isomorphism for any  $P^{\bullet}$ in $\K^{-}(\fP_n)$.
			\end{enumerate}
		\end{enumerate}
	\end{lemma}
	\begin{proof} (1) follows by Proposition \ref{Char-fn-exact complex}  and \cite[Lemma 2.4, Lemma 2.5]{CFH06}. (2) follows by Proposition \ref{Char-fn-quas-iso} and \cite[Proposition 2.6, Proposition 2.7]{CFH06}.
	\end{proof}

	\begin{proposition}\label{pure-quasi-iso-homot-equ} Let $f:X^{\bullet}\rightarrow Y^{\bullet}$ be an  $n$-quasi-isomorphism. If  $X^{\bullet}\in\K^{-}(\fP_n)$  and $Y^{\bullet}\in \K^{-}(\fP_n)$, then $f$ is a homotopy equivalence.
	\end{proposition}
	\begin{proof}
		Since $Y^{\bullet}\in \K^{-}(\fP_n)$, by Lemma \ref{char-pex-pquiso}(2), $\Hom_R(Y^{\bullet},f): \Hom_R(Y^{\bullet},X^{\bullet})\rightarrow \Hom_R(Y^{\bullet},Y^{\bullet})$ is a quasi-isomorphism. So we have an isomorphism of abelian groups:
		$$\Hh^0(\Hom_R(Y^{\bullet},f)): \Hh^0(\Hom_R(Y^{\bullet},X^{\bullet}))\rightarrow \Hh^0(\Hom_R(Y^{\bullet},Y^{\bullet})).$$
		So there is a cochain map $g:Y^{\bullet}\rightarrow X^{\bullet}$ such that $fg\sim \Id_{Y^{\bullet}}$. Similarly, there exists a cochain map: $h:X^{\bullet}\rightarrow Y^{\bullet}$ such that $gh\sim \Id_{X^{\bullet}}$. Hence, $g$ and $f$ are homotopy equivalences.
	\end{proof}

	\section{The relative derived categories induced by $\FPR$}
	Set
	\begin{center}
		$\K_{n}(R)=\{X^{\bullet}\in \K(R)\mid  X^{\bullet}$  is an $n$-exact complex$\}$.
	\end{center}
	Recall from \cite[Definition 2.1.6]{N01} that a subcategory $\mathscr{C}$ of  a triangulated category   $\mathscr{T}$ is called to be  thick if it is triangulated, and it contains all direct summands of its objects.
	
	\begin{lemma}\label{thick-n}
		$\K_{n}(R)$ is a  thick subcategory of $\K(R)$.  So the Verdier quotient $\K(R)/\K_{n}(R)$ is naturally  a triangulated category.
	\end{lemma}
	\begin{proof} Suppose that $f: X^{\bullet}\rightarrow Y^{\bullet}$ is a cochain map between $n$-exact complexes. Then
		$\Cone^{\bullet}(f)$ is again $n$-exact. Thus $\K_{n}(R)$  is a triangulated subcategory of $\K(R)$.
		Moreover, $\K_{n}(R)$  is closed under direct summands. Indeed, let $X^{\bullet}_1\oplus X^{\bullet}_2\in \K_{n}(R)$.
		Then $\Hom_R(P, X^{\bullet}_1\oplus X^{\bullet}_2)$ is exact for any $P\in \Proj^{<n+1}$, which implies that $\Hom_R(P, X^{\bullet}_1)$ and $\Hom_R(P, X^{\bullet}_2)$ are exact. Thus $X^{\bullet}_1$ and $X^{\bullet}_2$ are $n$-exact.
		Therefore, $\K_{n}(R)$ is a thick subcategory of  $\K(R)$. So the Verdier quotient $\K(R)/\K_{n}(R)$ is naturally  a triangulated category (see \cite[Theorem 2.1.8]{N01}).
	\end{proof}

	\begin{definition}{\rm
			Let $R$ be a ring. The Verdier quotient
			$$\D_{n}(R):=\K(R)/\K_{n}(R)$$
			is said to be the {\it $n$-derived category} of $R$.}
	\end{definition}
	
	For $\ast\in \{\mbox{blank},+,-,b\}$, set $\K^{\ast}_{n}(R)=\{X^{\bullet}\in \K^{\ast}(R)\mid  X^{\bullet}$  is an $n$-exact complex$\}$. We can define $$\D^{\ast}_{n}(R):=\K^{\ast}(R)/\K^{\ast}_{n}(R).$$
	
	Since $0$-exact complexes are exactly exact complexes,  $\D^{\ast}_{0}(R)$ of $R$  is exactly the derived categories $\D^{\ast}(R)$ of $R$.  So the notion of $n$-derived categories is a generalization of that of derived categories. Next, we will show how to construct $n$-derived categories.

	Every morphism from $X^{\bullet}$ to $Y^{\bullet}$ in $\D^{\ast}_{n}(R)$ is called a left roof and can be written  as the form $$X^{\bullet}\stackrel{s}{\Longleftarrow}\bullet\stackrel{a}{\longrightarrow} Y^{\bullet}$$
	where $s$ is an $n$-quasi-isomorphism. Two roofs $X^{\bullet}\stackrel{s}{\Longleftarrow}\bullet\stackrel{a}{\longrightarrow} Y^{\bullet}$ and $X^{\bullet}\stackrel{s_1}{\Longleftarrow}\bullet\stackrel{a_1}{\longrightarrow} Y^{\bullet}$ are equivalent if there exists the following commutative diagram
	
	$$\xymatrix@R=20pt@C=25pt{ &\ar@{=>}[ld]_{s}  \bullet \ar[rd]^{a}&\\
		X^{\bullet}& \ar@{=>}[l]_{s_2}   \bullet\ar[u]\ar[d]\ar[r]^{a_2} &Y^{\bullet}\\
		&  \ar@{=>}[lu]^{s_1} \bullet\ar[ru]_{a_1} & \\
	}$$
	where $s_2$ is an $n$-quasi-isomorphism. We denote by $a/s$ the equivalent class of $X^{\bullet}\stackrel{s}{\Longleftarrow}\bullet\stackrel{a}{\longrightarrow} Y^{\bullet}$.  Dually, the right roof $X^{\bullet}\stackrel{a}{\longrightarrow}\bullet\stackrel{s}{\Longleftarrow} Y^{\bullet}$ and its equivalent class $a\backslash s$ are defined.

	Next, we will characterize the left little finitistic dimension of the ring $R$ in terms of $n$-derived categories.
	For $\ast\in \{\mbox{blank},+,-,b\}$, set
	\begin{center}
		$\K^{\ast}_{n}(R)=\{X^{\bullet}\in \K^{\ast}(R)\mid  X^{\bullet}$  is an $n$-exact complex$\}$.
	\end{center}
	As a similar argument in the proof of Lemma \ref{thick-n}, one can show $\K^{\ast}_{m}(R)$ is a  thick subcategory of $\K^{\ast}_{n}(R)$ for any $m\geq n$. Set $\K^{\ast}_{n,m}(R)=\K^{\ast}_{n}(R)/\K^{\ast}_{m}(R)$ the  Verdier quotient of $\K^{\ast}_{n}(R)$ by $\K^{\ast}_{m}(R)$. Then $\K^{\ast}_{n,m}(R)$ is a  thick subcategory of $\D^{\ast}_{n}(R)$ for any $m>n$.
	
	\begin{theorem}\label{fpd-derived} Let $R$ be a ring. Then for any $\ast\in \{\mbox{blank},+,-,b\}$, there is a  triangulated equivalence for each $n$ and each $m\geq n$:
		$$\D^{\ast}_{n}(R)\cong \D^{\ast}_{m}(R)/\K^{\ast}_{n,m}(R).$$
		
		Moreover, the following statements are equivalents for each $n\in\mathbb{N}$:
		\begin{enumerate}
			\item $l.\fPD(R)\leq n$.
			\item $\K^{\ast}_{n,m}(R)=0$ for any $m>n$.
			\item $\K^{\ast}_{n,m}(R)=0$ for some $m>n$.
			\item $\D^{\ast}_{m}(R)$ is naturally triangulated equivalence to $\D^{\ast}_{n}(R)$ for any $m>n$.
			\item $\D^{\ast}_{m}(R)$ is naturally  triangulated equivalence to $\D^{\ast}_{n}(R)$ for some $m>n$.
		\end{enumerate}

	\end{theorem}
	\begin{proof} The triangulated equivalence $\D^{\ast}_{n}(R)\cong \D^{\ast}_{m}(R)/\K^{\ast}_{n,m}(R)$  follows by their definitions  and the universal properties of Verdier quotients. Next we will show the equivalences of $(1)-(5)$.
		
		$(2)\Rightarrow (3)$ and $(4)\Rightarrow (5)$ Trivial.  $(2)\Leftrightarrow (4)$ and $(3)\Leftrightarrow (5)$ follow by the triangulated equivalence $\D^{\ast}_{n}(R)\cong \D^{\ast}_{m}(R)/\K^{\ast}_{n,m}(R).$
		
		$(1)\Rightarrow (2)$ Suppose $l.\fPD(R)\leq n$. Then by Corollary \ref{fpd-exact-comp}, $\K^{\ast}_{n}(R)=\K^{\ast}_{m}(R)$ for any $m>n$. Thus $\K^{\ast}_{n,m}(R)=0$ for any $m>n$.

		$(3)\Rightarrow (1)$ Let $m>n$.  By Corollary \ref{fpd-exact-comp}, we just need to show that every $n$-exact complex is $m$-exact. Let $X^{\bullet}$ be an $n$-exact complex.  Since $\K^{\ast}_{n,m}(R)=0$, $\K^{\ast}_{n}(R)=\K^{\ast}_{m}(R)$. So $X^{\bullet}$ is homotopic to an $m$-exact complex, say  $Y^{\bullet}$. Let $P\in \fP_m$. Then $\Hom_R(P,X^{\bullet})$ is also homotopic to $\Hom_R(P,Y^{\bullet})$  which is an exact complex. Thus $\Hom_R(P,X^{\bullet})$ is also exact. So  $X^{\bullet}$ is an $n$-exact complex.
	\end{proof}

	We recall the notion of pure derived categories from \cite{ZZ16}. Set $\K_{\mathcal{PE}}(R)=\{X^{\bullet}\in \K(R)\mid  X^{\bullet}$  is a pure exact complex$\}$. Then $\K_{\mathcal{PE}}(R)$ is a thick subcategory of $\K(R)$. The pure derived category of $R$ is defined to be $$\D_{\textbf{pur}}(R)=\K(R)/\K_{\mathcal{PE}}(R).$$
	Similarly, the authors in \cite{ZZ16} defined $\K^{\ast}_{\mathcal{PE}}(R)$ and $\D^{\ast}_{\textbf{pur}}(R)$ for any $\ast\in \{\mbox{blank},+,-,b\}$. It is also easy to  show that $\K^{\ast}_{\mathcal{PE}}(R)$ is a  thick subcategory of $\K^{\ast}_{n}(R)$. Set $\K^{\ast}_{n,\textbf{pur}}(R)=\K^{\ast}_{n}(R)/\K^{\ast}_{\textbf{pur}}(R)$ the  Verdier quotient of $\K^{\ast}_{n}(R)$ by $\K^{\ast}_{\textbf{pur}}(R)$. Then $\K^{\ast}_{n,\textbf{pur}}(R)$ is a  thick subcategory of $\D^{\ast}_{n}(R)$.
	
	\begin{theorem} Let $R$ be a ring. Then for any $\ast\in \{\mbox{blank},+,-,b\}$, there is a  triangulated equivalence for each $n$:
		$$\D^{\ast}_{n}(R)\cong \D^{\ast}_{\textbf{pur}}(R)/\K^{\ast}_{n,\textbf{pur}}(R).$$
Furthermore, the following statements are equivalents:
		\begin{enumerate}
			\item $R$ is a left coherent ring  over which every finitely presented $R$-module has  projective dimension $<n+1$.
			\item $\K^{\ast}_{n,\textbf{pur}}(R)=0$.
			\item $\D^{\ast}_{\textbf{pur}}(R)$ is naturally  triangulated equivalence to $\D^{\ast}_{n}(R)$.
		\end{enumerate}
		Moreover, if $n\in\mathbb{N}$, then all above are equivalent to that   $R$ is a left coherent ring with $\cwd(R)<n+1$.
	\end{theorem}
	\begin{proof} It is similar with the proof of Theorem \ref{fpd-derived}, and so we omitted it.
	\end{proof}

	\begin{proposition}\label{prop-K-D}  Let  $X^{\bullet}\in\K^{-}_{n}(\fP_n)$ and $Y^{\bullet}\in \K(R)$. Then the localization functor $$\mathbb{F}: \Hom_{\K(R)}(X^{\bullet},Y^{\bullet})\rightarrow  \Hom_{\D_{n}(R)}(X^{\bullet},Y^{\bullet}),\ f\mapsto f/\Id_{X^{\bullet}} $$
			induces an isomorphism of abelian groups.
\end{proposition}
	\begin{proof} Suppose $f/\Id_{X^{\bullet}}=0=0/\Id_{X^{\bullet}}$. Then there is an $n$-quasi-isomorphism $g:Z^{\bullet}\rightarrow X^{\bullet}$ such that $gf\sim 0$.
		By the proof of Proposition \ref{pure-quasi-iso-homot-equ}, there exists an $n$-quasi-isomorphism  $h:X^{\bullet}\rightarrow Z^{\bullet}$ such that $gh\sim\Id_{X^{\bullet}}$. Hence $f\sim 0$, and so $\mathbb{F}$ is injective.
		Let $f/s\in \Hom_{\D_{n}(R)}(X^{\bullet},Y^{\bullet})$. By the proof of Propsotion \ref{pure-quasi-iso-homot-equ} again, there exits an $n$-quasi-isomorphism $t$ such that $st\sim \Id_{X^{\bullet}}$. So $f/s=(ft)/\Id_{X^{\bullet}}\in \D_{n}(R)$, and thus $\mathbb{F}$ is surjective.
	\end{proof}

	\begin{lemma}\label{fnpure-homotopy} The following two statements hold.
		\begin{enumerate}
			\item Let $Y^{\bullet}\rightarrow X^{\bullet}$ be an  $n$-quasi-isomorphism with $X^{\bullet}\in\K^b(R),Y^{\bullet}\in \K^+(R)$. Then there exists an $n$-quasi-isomorphism $Z^{\bullet}\rightarrow Y^{\bullet}$ with $Z^{\bullet}\in\K^b(R)$.
			\item Let $X^{\bullet}\rightarrow Y^{\bullet}$ be an  $n$-quasi-isomorphism with $X^{\bullet}\in\K^+(R),Y^{\bullet}\in \K(R)$. Then there exists an $n$-quasi-isomorphism $Y^{\bullet}\rightarrow Z^{\bullet}$ with $Z^{\bullet}\in\K^+(R)$.
		\end{enumerate}
	\end{lemma}
	\begin{proof} (1) Without loss of generality, we may assume that $\Hh^i(M\otimes_RY^{\bullet})=0$ for any $i<0$ and $\Hh^i(\Hom_R(P,Y^{\bullet}))=0$ for any $n$-projective $R$-modules $P$ and any $i<0$. Then we have the following commutative diagram:
		$$\xymatrix@R=20pt@C=20pt{
			\cdots\ar[r]^{}&0 \ar@{=}[d]^{}\ar[r]^{} & Y^0\ar@{=}[d]\ar[r]^{} & \cdots\ar[r]^{} &Y^{m-1}\ar@{=}[d]\ar[r]^{} & \Ker(d_{Y^{\bullet}}^m)\ar[d]^{}\ar[r]^{}& 0\ar[d]^{}\ar[r]^{}&\cdots \\
			\cdots\ar[r]^{}&0 \ar[r]^{} & Y^0\ar[r]^{} & \cdots \ar[r]^{} &Y^{m-1}\ar[r]^{} & Y^{m}\ar[r]^{} &Y^{m+1}\ar[r]^{} &\cdots\\}$$
		Define the complex $Z^{\bullet}$ to be the upper row. Since $\Hom_R(P,-)$ preserves kernels, the cochain map above is clear an $n$-quasi-isomorphism by Proposition \ref{Char-fn-quas-iso}(2).

		(2) Without loss of generality, we may assume that $Y^i=0$ for any $i<0$ and $\Hh^i(\Hom_R(P,Y^{\bullet}))=0$ for any $n$-projective $R$-modules $P$ and any $i\geq m+1$. Then we have the following commutative diagram:
		$$\xymatrix@R=20pt@C=20pt{
			\cdots\ar[r]^{}&Y^{-2}\ar[d] \ar[r]^{} & Y^{-1}\ar[d]^{}\ar[r]^{} & Y^0\ar[d]\ar[r]^{} &Y^{1}\ar@{=}[d]\ar[r]^{} & Y^2\ar@{=}[d]\ar[r]^{}& \cdots \\
			\cdots\ar[r]^{}&0 \ar[r]^{} & 0\ar[r]^{} & \Coker(d_{Y^{\bullet}}^{-1})\ar[r]^{} &Y^{1}\ar[r]^{} & Y^{2}\ar[r]^{} &\cdots\\}$$
		Define the complex $Z^{\bullet}$ to be the lower row. Since $P\otimes_R-$ preserves cokernels, the cochain map above is clear an $n$-quasi-isomorphism by Proposition \ref{Char-fn-quas-iso}(1).
	\end{proof}

	\begin{proposition}\label{5.7}   The following two statements hold.
		\begin{enumerate}
			\item  $\D^b_{n}(R)$ is a full subcategory of $\D^+_{n}(R)$, and  $\D^+_{n}(R)$ is a full subcategory of $\D_{n}(R)$.
			\item   $\D^b_{n}(R)$ is a full subcategory of $\D^-_{n}(R)$, and  $\D^-_{n}(R)$ is a full subcategory of $\D_{n}(R)$.
			\item    $\D^+_{n}(R)\cap \D^-_{n}(R)=\D^b_{n}(R)$.
		\end{enumerate}
	\end{proposition}
	\begin{proof} (1) is a consequence of \cite[Proposition 3.2.10]{GM03} and Lemma \ref{fnpure-homotopy}. (2) is the dual of $(1)$, and $(3)$ is an immediate consequence of (1) and (2).
	\end{proof}

	\begin{theorem}\label{fully faithful}    $R$-Mod is a full subcategory of $\D^b_{n}(R)$. That is the composition of functors
		\begin{center}
			$R$-\Mod$\rightarrow \K^b(R)\rightarrow \D^b_{n}(R)$
		\end{center}
		is fully faithful.
	\end{theorem}
	\begin{proof} For any $X,Y\in R$-Mod, it suffice to prove
		\begin{equation*}
			\begin{aligned}
				\mathbb{F}:\ & \Hom_R(X,Y)\rightarrow \Hom_{\D^b_{n}}(X,Y)\\
				&f\mapsto f/\Id_X
			\end{aligned}
		\end{equation*}
		is an isomorphism.
		Let $f\in \Hom_R(X,Y)$ such that $f/\Id_X=0=0/\Id_X$. Then there exists an $n$-quasi-isomorphism $s:Z\rightarrow X$ such that $fs\sim 0$, and hence $\Hh^0(f)\Hh^0(s)=0$. Since $\Hh^0(s)$ is an isomorphism, we have $f=\Hh^0(f)=0$. So $\mathbb{F}$ is injective.
		
		Let $X^{\bullet}\stackrel{s}{\Longleftarrow}Z^{\bullet}\stackrel{a}{\longrightarrow} Y^{\bullet}\in \Hom_{\D^b_{n}}(X,Y)$, where $s$ is an $n$-quasi-isomorphism.
		Then $\Hh^0(s)\in \Hom_R(\Hh^0(Z^{\bullet}),X)$
		is an isomorphism of $R$-modules. Put $f:=\Hh^0(a)\Hh^0(s)^{-1}\in \Hom_R(X,Y)$. Consider the truncation
		$$\tau_{\leq 0}Z^{\bullet}:=\ \cdots\longrightarrow Z^{-2}\longrightarrow Z^{-1}\longrightarrow \Ker(d^0_{Z^{\bullet}})\longrightarrow0$$
		of $Z^{\bullet}$ and the canonical inclusion $i:\tau_{\leq 0}Z^{\bullet}\rightarrow Z^{\bullet}$. Note that $i$ is an $n$-quasi-isomorphism as in the proof of Lemma \ref{fnpure-homotopy}. Consider the following commutative diagram
		$$\xymatrix@R=20pt@C=25pt{ &\ar@{=>}[ld]_{s}  Z^{\bullet} \ar[rd]^{a}&\\
			X& \ar@{=>}[l]_{si}   \tau_{\leq 0}Z^{\bullet}\ar[u]\ar[d]\ar[r]^{ai} &Y\\
			&  \ar@{=>}[lu]^{\Id_{X^{\bullet}}}  X^{\bullet}\ar[ru]_{si} & \\
		}$$
		It follows that $f/\Id_{X^{\bullet}}=a/s$, and so $\mathbb{F}$ is surjective.
	\end{proof}

	For any $X^{\bullet}\in \C(R)$, define
	\begin{equation*}
		\begin{aligned}
			&\boldsymbol{inf_n}X^{\bullet}:=\inf\{m\in\mathbb{Z}\mid X^{\bullet}\ \mbox{is not}\  n\mbox{-exact at}\
			m\}, and\\
			&\boldsymbol{sup_n}X^{\bullet}:=\sup\{m\in\mathbb{Z}\mid X^{\bullet} \ \mbox{is not}\  n\mbox{-exact at}\
			m\}
		\end{aligned}
	\end{equation*}
	If $X^{\bullet}$ is not $n$-exact at any $i\in\mathbb{Z}$, then we set $\boldsymbol{inf_n}X^{\bullet}=-\infty$ and $\boldsymbol{sup_n}X^{\bullet}=\infty$. If $X^{\bullet}$ is  $n$-exact at all $i\in\mathbb{Z}$, then we set $\boldsymbol{inf_n}X^{\bullet}=\infty$ and $\boldsymbol{sup_n}X^{\bullet}=-\infty$.
	
	Set \begin{equation*}
		\begin{aligned}
			&\K^{-,\boldsymbol{nb}}(\fP_n):=\{X^{\bullet}\in\K^{-}(\fP_n)\mid \boldsymbol{inf_n}X^{\bullet}\ \mbox{is finite} \}
		\end{aligned}
	\end{equation*}

	\begin{proposition}\label{prop-K} Let $X^{\bullet}\in \C(R)$.   Then the following statements hold.
		\begin{enumerate}
			\item  $X^{\bullet}$ is $n$-exact in degree $\geq i$ if and only if $\Hom_R(P, X^{\bullet})$ is exact in degree $\geq i$ for any $n$-projective module $P$.
			\item    The numbers $\boldsymbol{inf_n}X^{\bullet}$ and $\boldsymbol{sup_n}X^{\bullet}$ are well-defined for any $X^{\bullet}\in \D_{n}(R)$.
			\item   $\K^{-,\boldsymbol{nb}}(\fP_n)$ is a  triangulated subcategory of $\K^{-}(\fP_n)$.
		\end{enumerate}
	\end{proposition}
	\begin{proof} (1) follows by the  commutative diagram:
		$$\xymatrix@R=20pt@C=8pt{
			\cdots\ar[r]^{}&\Hom_R(P,X^{m-1})\ar@{->>}[rd]_{\Hom_R(P,\pi^{m-2}_{X^{\bullet}})} \ar[rr]^{\Hom_R(P,d^{m-1}_{X^{\bullet}})} && \Hom_R(P,X^{m})\ar@{->>}[rd]_{\Hom_R(P,\pi^{m-1}_{X^{\bullet}})}\ar[rr]^{\Hom_R(P,d^{m}_{X^{\bullet}})}& & \Hom_R(P,X^{m}) \ar[r]^{} & \cdots \\
			&& \Hom_R(M,C^{m-2})\ar[ru]_{\Hom_R(P,i^{m}_{X^{\bullet}})} &&\Hom_R(M,C^{m-1})\ar[ru]_{\Hom_R(P,i^{m+1}_{X^{\bullet}})} &  &\\}$$
		where $\pi_{X^{\bullet}}^{m}$ denotes the cokernel of $d^m_{X^{\bullet}}$ and $i_{X^{\bullet}}^{m}$ denotes the kernel of $d^m_{X^{\bullet}}$ for any $m\in \mathbb{Z}$.

		(2) We only need to prove the assertion whenever both $\boldsymbol{inf_n}X^{\bullet}$ (resp., $\boldsymbol{sup_n}X^{\bullet}$) and $\boldsymbol{inf_n}Y^{\bullet}$ (resp., $\boldsymbol{sup_n}Y^{\bullet}$)  are finite. This is the consequence of  (1)and (2) by Proposition \ref{Char-fn-quas-iso}.
		
		(3) We only prove  $\K^{-,\boldsymbol{nb}}(\fP_n)$ is a triangulated subcategories of $\K^{-}(\fP_n)$ since the other assertion can be proved similarly. Obviously, $\K^{-,\boldsymbol{nb}}(\fP_n)$  is closed under shifts. So it suffices to show that  $\K^{-,\boldsymbol{nb}}(\fP_n)$ is closed under extensions. Let $$X^{\bullet}\rightarrow Y^{\bullet}\rightarrow Z^{\bullet}\rightarrow X^{\bullet}[1]$$
		be a triangle in $\K^{-}(\fP_n)$ with $X^{\bullet},Z^{\bullet}\in \K^{-,\boldsymbol{nb}}(\fP_n)$. Then we have a triangle
		$$ \Hom_R(P,X^{\bullet})\rightarrow \Hom_R(P,Y^{\bullet})\rightarrow \Hom_R(P,Z^{\bullet})\rightarrow \Hom_R(P,X^{\bullet}[1])$$
		in $\K(\mathbb{Z})$ for any $n$-projective $R$-module $P$. It induces the following long exact sequence
		$$\cdots\rightarrow\Hh^i(\Hom_R(P,X^{\bullet}))\rightarrow \Hh^i(\Hom_R(P,Y^{\bullet}))\rightarrow \Hh^i(\Hom_R(P,Z^{\bullet}))\rightarrow \Hh^i(\Hom_R(P,X^{\bullet}[1]))\rightarrow\cdots$$
		since $\Hh^i(-)$ is a cohomological functor. By (1), there exists $m\in\mathbb{Z}$ such that $\Hh^i(\Hom_R(P,Y^{\bullet}))=\Hh^i(\Hom_R(P,Z^{\bullet}))=0$ for any $i\leq m$. So $\Hh^i(\Hom_R(P,Y^{\bullet}))=0$  for any $i\leq m$. Thus we have $Y^{\bullet}\in \K^{-,\boldsymbol{nb}}(\fP_n)$ by (1).
	\end{proof}

	\begin{proposition}\label{5.10} There exists a triangulated functor $\mathbb{P}:\K^b(R)\rightarrow \K^{-,\boldsymbol{nb}}(\fP_n)$ and an $n$-quasi-isomorphism $f_{X^{\bullet}}: \mathbb{P}_{X^{\bullet}}\rightarrow X^{\bullet}$ for any $X^{\bullet}\in \K^b(R)$ which is functorial in $X^{\bullet}$.
	\end{proposition}
	\begin{proof}
		We first prove that for any  $X^{\bullet}\in \K^b(R)$ there exists an $n$-quasi-isomorphism $\mathbb{P}_{X^{\bullet}}\rightarrow X^{\bullet}$ with $\mathbb{P}_{X^{\bullet}}\in \K^{-,\boldsymbol{nb}}(\fP_n)$. We proceed by induction on the cardinal of finite set $\kappa(X^{\bullet}):=\{i\in\mathbb{Z}\mid X^i\not=0\}$.
		
		If cardinal of $\kappa(X^{\bullet})$ is equal to $1$, then the assertion follows from that every $R$-module has an $n$-projective precover (see Theorem \ref{pre-c-e-fpn}).
		
		Now suppose that the cardinal of $\kappa(X^{\bullet})$ is lager than $1$. We may assume that $X^{i}=0$ for any $i<j$. Set $X^{\bullet}_1=X^j[-j-1]$ and $X^{\bullet}_2=\cdots\rightarrow 0\rightarrow X^{j+1}\rightarrow X^{j+2}\rightarrow \cdots$. Then we get a sequence of complexes
		$$\xymatrix@R=20pt@C=20pt{
			X^{\bullet}_2\ar[d]^{}  &&\cdots\ar[r]^{}&0 \ar[r]^{}\ar[d]^{} & 0\ar[d]^{}\ar[r]^{} & X^{j+1}\ar[r]\ar@{=}[d] & X^{j+2}\ar@{=}[d]\ar[r]^{} &\cdots \\
			X^{\bullet}\ar[d]^{} &&\cdots\ar[r]^{}&0 \ar[r]^{}\ar[d]^{} & X^j\ar@{=}[d]\ar[r]^{} &  X^{j+1}\ar[r]^{} \ar[d]^{} &X^{j+2}\ar[d]^{}\ar[r]^{} &\cdots \\
			X^{\bullet}_1[1]&&\cdots\ar[r]^{}&0 \ar[r]^{} &  X^j\ar[r]^{} & 0 \ar[r]^{} &0\ar[r]^{} &\cdots\\}$$
		which is split in each degree. Thus we have a triangle
		$$X^{\bullet}_1\xrightarrow{u} X^{\bullet}_2\rightarrow X^{\bullet}\rightarrow X^{\bullet}_1[1]$$
		in $\K^b(R)$. By induction, there exist $n$-quasi-isomorphisms $f_{X_1^{\bullet}}:\mathbb{P}_{X_1^{\bullet}}\rightarrow X_1^{\bullet}$ and $f_{X_2^{\bullet}}:\mathbb{P}_{X_2^{\bullet}}\rightarrow X_2^{\bullet}$ with $\mathbb{P}_{X_1^{\bullet}}$ and $\mathbb{P}_{X_2^{\bullet}}$ in $\K^{-,\boldsymbol{nb}}(\fP_n)$. Then by Lemma \ref{char-pex-pquiso}, ${X_2^{\bullet}}$ induces a quasi-isomorphism of complexes $$\Hom_R(\mathbb{P}_{X_1^{\bullet}},\mathbb{P}_{X_2^{\bullet}})\cong \Hom_R(\mathbb{P}_{X_1^{\bullet}},X_2^{\bullet}).$$
		Thus there is a morphism $f:\mathbb{P}_{X_1^{\bullet}}\rightarrow \mathbb{P}_{X_2^{\bullet}}$ which is unique up to homotopy (see  Proposition \ref{pure-quasi-iso-homot-equ}), such that $f_{X_2^{\bullet}}f=uf_{X_1^{\bullet}}$. Consider the triangle $$\mathbb{P}_{X_1^{\bullet}}\xrightarrow{f} \mathbb{P}_{X_2^{\bullet}}\rightarrow \Cone^{\bullet}(f)\rightarrow \mathbb{P}_{X_1^{\bullet}}[1]$$
		in $\K^b(R)$. Since $\mathbb{P}_{X_1^{\bullet}},\mathbb{P}_{X_2^{\bullet}}\in \K^{-,\boldsymbol{nb}}(\fP_n)$, $\Cone^{\bullet}(f)\in \K^{-,\boldsymbol{nb}}(\fP_n)$ by Proposition \ref{prop-K}, and hence the triangle is also in $\K^{-,\boldsymbol{nb}}(\fP_n)$.  So there is a morphism ${X^{\bullet}}:\Cone^{\bullet}(f)\rightarrow X^{\bullet}$ such that the following diagram commutates in $\K^b(R)$
		$$\xymatrix@R=25pt@C=30pt{
			\mathbb{P}_{X_1^{\bullet}}\ar[d]^{f_{X_1^{\bullet}}}\ar[r]^{f}&\mathbb{P}_{X_2^{\bullet}}\ar[r]^{}\ar[d]^{f_{X_2^{\bullet}}}&\Cone^{\bullet}(f)\ar[r]^{}\ar[d]^{f_{X^{\bullet}}}&\mathbb{P}_{X_1^{\bullet}}[1] \ar[d]^{f_{X_1^{\bullet}[1]}}\\
			X_1^{\bullet}\ar[r]^{u}&X_2^{\bullet}\ar[r]^{}&X^{\bullet}\ar[r]^{}&X_1^{\bullet}[1]\\}$$
		So ${X^{\bullet}}$ is also an $n$-quasi-isomorphism by Corollary \ref{cor-fn-p-qiso}.

		Put $\mathbb{P}_{X^{\bullet}}:=\Cone^{\bullet}(f)$. Then there is an  $n$-quasi-isomorphism $f_{X^{\bullet}}:\mathbb{P}_{X^{\bullet}}\rightarrow X^{\bullet}$  with $\mathbb{P}_{X^{\bullet}}\in \K^{-,\boldsymbol{nb}}(\fP_n)$.
		Next, we will show $f_{X^{\bullet}}$ is functorial in $X^{\bullet}$. Let $X^{\bullet},Y^{\bullet}\in \K^b(R)$. We have two  $n$-quasi-isomorphisms
		$f_{X^{\bullet}}:\mathbb{P}_{X^{\bullet}}\rightarrow X^{\bullet}$  and $f_{Y^{\bullet}}:\mathbb{P}_{Y^{\bullet}}\rightarrow Y^{\bullet}$. Then $f_{Y^{\bullet}}$ induces an isomorphism
		$$\Hom_{\K(R)}(\mathbb{P}_{X^{\bullet}},f_{Y^{\bullet}}):\Hom_{\K(R)}(\mathbb{P}_{X^{\bullet}},\mathbb{P}_{Y^{\bullet}})\rightarrow \Hom_{\K(R)}(\mathbb{P}_{X^{\bullet}},Y^{\bullet}).$$
		Let $f:X^{\bullet}\rightarrow Y^{\bullet}$ be a cochain map, then there is a unique cochain map $\mathbb{P}_f:\mathbb{P}_{X^{\bullet}}\rightarrow {Y^{\bullet}}$ such that the following diagram
		$$\xymatrix@R=25pt@C=30pt{
			\mathbb{P}_{X^{\bullet}}\ar[d]^{{X^{\bullet}}}\ar[r]^{\mathbb{P}_f}&\mathbb{P}_{Y^{\bullet}}\ar[d]^{{Y^{\bullet}}}\\
			X^{\bullet}\ar[r]^{f}&Y^{\bullet}\\}$$
		commutates up to homotopy. By putting $f=\Id_{X^{\bullet}}:X^{\bullet}\rightarrow X^{\bullet}$, we complete the proof.
		
		Finally, we will prove $\mathbb{P}:\K^b(R)\rightarrow \K^{-,\boldsymbol{nb}}(\fP_n)$  is a  triangulated functor.
		Let $X^{\bullet}\xrightarrow{f} Y^{\bullet}\xrightarrow{g} Z^{\bullet}\xrightarrow{h} X^{\bullet}[1]$ be a triangle in $\K^b(R)$. Then we have a triangle $$\mathbb{P}_{X^{\bullet}}\xrightarrow{\mathbb{P}_{f}} \mathbb{P}_{Y^{\bullet}}\xrightarrow{u} \Cone^{\bullet}(\mathbb{P}_{f})\xrightarrow{v} X^{\bullet}[1]$$
		Consider the following commutative diagram
		$$\xymatrix@R=25pt@C=30pt{
			\mathbb{P}_{X^{\bullet}}\ar[d]^{f_{X^{\bullet}}}\ar[r]^{\mathbb{P}_{f}}&\mathbb{P}_{Y^{\bullet}}\ar[r]^{u}
			\ar[d]^{f_{Y^{\bullet}}}&\Cone^{\bullet}(\mathbb{P}_{f})\ar[r]^{v}\ar@{.>}[d]^{\gamma}
			&\mathbb{P}_{X^{\bullet}}[1] \ar[d]^{f_{X^{\bullet}[1]}}\\
			X^{\bullet}\ar[r]^{f}&Y^{\bullet}\ar[r]^{g}&Z^{\bullet}\ar[r]^{h}&X^{\bullet}[1]\\}$$
		Since ${X^{\bullet}}$ and ${Y^{\bullet}}$ are $n$-quasi-isomorphisms, $\gamma$ is also an  $n$-quasi-isomorphism by Corollary \ref{cor-fn-p-qiso}.
		And $\gamma$ induces an isomorphism of Abelian groups:
		$$\Hom_{\K(R)}(\mathbb{P}_{Z^{\bullet}},\gamma): \Hom_{\K(R)}(\mathbb{P}_{Z^{\bullet}},\Cone^{\bullet}(\mathbb{P}_{f}))\rightarrow \Hom_{\K(R)}(\mathbb{P}_{Z^{\bullet}},Z^{\bullet}).$$
		Let $f_{Z^{\bullet}}:\mathbb{P}_{Z^{\bullet}}\rightarrow Z^{\bullet}$ defined above. Then there exists $w:\mathbb{P}_{Z^{\bullet}}\rightarrow \Cone^{\bullet}(\mathbb{P}_{f})$ such that $\gamma=f_{Z^{\bullet}}w$. Since $\mathbb{P}_{Z^{\bullet}}$ and $\Cone^{\bullet}(\mathbb{P}_{f})$ are in $\K^{-,\boldsymbol{nb}}(\fP_n)$, $w$ is actually a homotopy equivalence by Proposition \ref{pure-quasi-iso-homot-equ}.  And so $\Hom_{\K(R)}(\mathbb{P}_{Y^{\bullet}},\gamma):\Hom_{\K(R)}(\mathbb{P}_{Y^{\bullet}},\Cone^{\bullet}(\mathbb{P}_{f}))\rightarrow \Hom_{\K(R)}(\mathbb{P}_{Y^{\bullet}},Z^{\bullet})$ is an isomorphism. Thus  $\Hom_{\K(R)}(\mathbb{P}_{Y^{\bullet}},\gamma)(u)=\gamma u=gf_{Y^{\bullet}}=f_{Z^{\bullet}}\mathbb{P}_{g}=\gamma w \mathbb{P}_{g}=\Hom_{\K(R)}(\mathbb{P}_{Y^{\bullet}},\gamma)(w\mathbb{P}_g)$. So $u=w\mathbb{P}_g$. Similarly, $\mathbb{P}_h=vw$. Thus we have the following commutative diagram:
		$$\xymatrix@R=25pt@C=30pt{
			\mathbb{P}_{X^{\bullet}}\ar@{=}[d]^{}\ar[r]^{\mathbb{P}_{f}}&\mathbb{P}_{Y^{\bullet}}\ar[r]^{\mathbb{P}_{g}}
			\ar@{=}[d]^{}&\mathbb{P}_{Z^{\bullet}}\ar[r]^{\mathbb{P}_{h}}\ar[d]^{w}
			&\mathbb{P}_{X^{\bullet}}[1] \ar@{=}[d]^{}[1]\\
			\mathbb{P}_{X^{\bullet}}\ar[r]^{\mathbb{P}_{f}}&\mathbb{P}_{Y^{\bullet}}\ar[r]^{u}&\Cone^{\bullet}(\mathbb{P}_{f})\ar[r]^{v}&\mathbb{P}_{X^{\bullet}}[1]\\}$$
		which shows that the top row is a triangle in $\K^{-,\boldsymbol{nb}}(\fP_n)$.
	\end{proof}
	
	\begin{theorem} There is a triangulated equivalence
		\begin{center}
			 $\D^b_{n}(R)\cong \K^{-,\boldsymbol{nb}}(\fP_n)$.
		\end{center}
	\end{theorem}
	\begin{proof}
		Let $\mathbb{H}$ be the composition $$\K^{-,\boldsymbol{nb}}(\fP_n)\hookrightarrow \K^{-}(R)\xrightarrow{\mathbb{F}}\D^-_{n}(R)$$
		of triangle functors, where $\mathbb{F}$ is the localization functor. For any $X^{\bullet}\in \K^{-,\boldsymbol{nb}}(\fP_n)$, there exists $m\in \mathbb{Z}$ such that $\boldsymbol{inf_n}X^{\bullet}=m$. Then $X^{\bullet}$ is $n$-exact in degree $\leq m-1$. Consider the following commutative diagram
		$$\xymatrix@R=20pt@C=15pt{
			X^{\bullet}\ar[d]^{f}&\cdots\ar[r]^{}&X^{m-2}\ar[d] \ar[r]^{} & X^{m-1}\ar[d]^{}\ar[r]^{} & X^m\ar@{->>}[d]\ar[r]^{} &X^{m+1}\ar@{=}[d]^{}\ar[r]^{} & X^{m+2}\ar@{=}[d]^{}\ar[r]^{}& \cdots \\
			X_{\supset m}^{\bullet}&\cdots\ar[r]^{}&0 \ar[r]^{} & 0\ar[r]^{} & \Coker(d_{X^{\bullet}}^{m-1})\ar[r]^{} &X^{m+1}\ar[r]^{} & X^{m+2}\ar[r]^{} &\cdots\\}$$
		
		Note that the cochain map $f$ is an $n$-quasi-isomorphism. It follows that $\mathbb{H}(X^{\bullet})\cong X_{\supset m}^{\bullet}$  in $\D_{n}(R)$. So $\mathbb{H}(X^{\bullet})\in \D_{n}(R)$ as $X^{\bullet}\in \K^{-,\boldsymbol{nb}}(\fP_n)$. Hence $\mathbb{H}$ induces a functor from $\K^{-,\boldsymbol{nb}}(\fP_n)$ to $\D^b_{n}(R)$ which can also be denoted by $\mathbb{H}$. By Theorem \ref{fully faithful}, $\mathbb{H}$ is fully faithful and dense. So $\D^b_{n}(R)\cong \K^{-,\boldsymbol{nb}}(\fP_n)$.
	\end{proof}

	\begin{remark}
We must remark that if $R$ is a commutative ring, then the ``injective''  versions of all results in this section also hold.
	\end{remark}

	\section{The relation between derived categories and $n$-derived categories}

Note that
	$$\K^{-,\boldsymbol{nb}}(\fP_n)= \left\{X^{\bullet}\in K^{-}(\fP_n) \, \vline \,\begin{matrix}  \text{there exists }
		m\in\mathbb{Z}, \text{ such that }\\
		\mathrm{H}^{i}(\mathrm{Hom}_{R}(P,X^{\bullet})) = 0, \forall i\leq m, \forall P\in\fP_n
	\end{matrix} \right\}.$$
	
	We denote by $\K^{-,b}(\mathscr{P})$ the homotopy category of upper bounded complexes of projective modules with only finitely many non-zero cohomologies.
	
	\begin{lemma}\label{5.2} Let $X^{\bullet}\in \K^{-,b}(\mathscr{P})$. Then there exists a quasi-isomorphism $X^{\bullet} \rightarrow P^{\bullet}$ with $P^{\bullet}\in \K^{-,\boldsymbol{nb}}(\fP_n)$.
	\end{lemma}
	\begin{proof}
		Since $X^{\bullet}\in \K^{-,b}(\mathscr{P})$, there is an integer $m$ such that $\Hh_i(X^{\bullet}) = 0$ for $i\leq m.$ Since $\fP_n$ is precovering by Theorem \ref{pre-c-e-fpn}, we can choose an $n$-projective resolution $\rightarrow P^{m-2}\rightarrow P^{m-1}\rightarrow \Ker(d_{X^{\bullet}}^m)\rightarrow 0$ for $\Ker(d_{X^{\bullet}}^m)$. By the Comparison Theorem
		(\cite[Theorem 6.16]{R09}), we have a cochain map
		$$\xymatrix@R=20pt@C=20pt{
			X^{\bullet}\ar[d]^{f}&&\cdots\ar[r]^{}&X^{m-2}\ar[d] \ar[r]^{} & X^{m-1}\ar[d]^{}\ar[r]^{} & X^m\ar@{=}[d]\ar[r]^{} &X^{m+1}\ar@{=}[d]^{}\ar[r]^{} & X^{m+2}\ar@{=}[d]^{}\ar[r]^{}& \cdots \\
			P^{\bullet}&&\cdots\ar[r]^{}&P^{m-2} \ar[r]^{} & P^{m-1}\ar[r]^{} & X^m\ar[r]^{} &X^{m+1}\ar[r]^{} & X^{m+2}\ar[r]^{} &\cdots\\}$$
		Clearly, $f$ is a quasi-isomorphism and  $P^{\bullet}\in \K^{-,\boldsymbol{nb}}(\fP_n)$.
	\end{proof}

	\begin{lemma}\label{5.3} Let  $X^{\bullet}\in \K^{-,b}(\mathscr{P})$ and $G^{\bullet}\in \K^{-,\boldsymbol{nb}}(\fP_n)$. Then for any cochain map
		$g : X^{\bullet} \rightarrow  G^{\bullet}$, there is a quasi-isomorphism $f : X^{\bullet}\rightarrow  G^{\bullet}$ with $P^{\bullet}\in \K^{-,\boldsymbol{nb}}(\fP_n)$, and a
		cochain map $h : P^{\bullet} \rightarrow G^{\bullet}$, such that $g\sim hf$.
	\end{lemma}
	\begin{proof} Since $X^{\bullet}\in \K^{-,b}(\mathscr{P})$, there exists an integer $m(X^{\bullet})$ such that $\Hh^i
		(\Hom_R(M, X^{\bullet})) = 0$ for any $i\leq m(X^{\bullet})$ and any projective module $M$. Similarly, since $G^{\bullet}\in \K^{-,\boldsymbol{nb}}(\fP_n)$,
		there exists an integer $m(G^{\bullet})$ such that $\Hh_j (\Hom_R(N, G^{\bullet}))=0$ for any $j \leq m(G^{\bullet})$ and any $N\in \fP_n.$ Let $m = \min\{m(X^{\bullet}), m(G^{\bullet})\}$. By Lemma \ref{5.2}, we get a quasi-isomorphism
		$f : X^{\bullet}\rightarrow  P^{\bullet}$, where $P^{\bullet}\in \K^{-,\boldsymbol{nb}}(\fP_n)$ with $P_i = X_i$ for $i\geq m$.
		For any $i\geq m$, let $f^i = \Id_{X^i}$ , and let $h^i = g^i$
		. Since $G^{\bullet}\in \K^{-,\boldsymbol{nb}}(\fP_n)$, the sequence
		$$0\rightarrow \Ker(d^{m-1}_{G^{\bullet}})\rightarrow G^{m-1}\rightarrow \Ker(d^{m}_{G^{\bullet}})\rightarrow  0$$
		is $n$-exact. In particular, we have an exact sequence
		$$0\rightarrow\Hom_R(P^{m-1}, \Ker(d^{m-1}_{G^{\bullet}}))\rightarrow \Hom_R(P^{m-1},G^{m-1})\rightarrow \Hom_R(P^{m-1},\Ker(d^{m}_{G^{\bullet}}))\rightarrow  0$$
		It follows from
		$$d^m_{G^{\bullet}} h^md^{m-1}_{P^{\bullet}}=d^m_{G^{\bullet}} g^md^{m-1}_{P^{\bullet}}= g^{m+1}d^m_{P^{\bullet}} d^{m-1}_{P^{\bullet}} = 0$$
		that $h^md^{m-1}_{P^{\bullet}}\in \Hom_R(P^{m-1}, \Ker(d^m_{G^{\bullet}}) ).$ This yields an $R$-homomorphism $h^{m-1}\in \Hom_R(P^{m-1},G^{m-1})$ such that $d^{m-1}_{G^{\bullet}} h^{m-1} = h^md^{m-1}_{P^{\bullet}} $. Inductively, we get $R$-homomorphisms $h_j : P^j\rightarrow G^j(j<m)$ such that $d^{j-1}_{G^{\bullet}}h^{j-1} = h^jd^{j-1}_{P^{\bullet}}$
		. Hence we have a cochain map $h : P^{\bullet}\rightarrow G^{\bullet}.$
		
		Now consider the following diagram:
		
		$$\xymatrix@R=25pt@C=12pt{
			X^{\bullet}_2\ar[d]^{}  &\cdots\ar[r]^{}&X^{m-2} \ar[rr]^{}\ar[dd]^{g^{m-2}} \ar[rd]^{f^{m-2}} && X^{m-1}\ar[dd]^{g^{m-1}}\ar[rd]^{f^{m-1}}\ar[rr]^{} && X^{m}\ar[rr]\ar[dd]^{g^{m}}\ar[rd]^{\Id} && X^{m+1}\ar[dd]^{g^{m+1}}\ar[rd]^{\Id}\ar[r]^{} &\cdots \\
			P^{\bullet}\ar[d]^{} &\cdots\ar[rr]^{}&&P^{m-2} \ar[rr]^{}\ar@{-->}[ld]^{h^{m-2}} && P^{m-1}\ar@{-->}[ld]^{h^{m-1}}\ar[rr]^{} &&  X^{m}\ar[rr]^{} \ar@{-->}[ld]^{h^{m}} &&X^{m+1}\ar@{-->}[ld]^{h^{m+1}}\ar[r]^{} &\cdots \\
			G^{\bullet}&\cdots\ar[r]^{}&G^{m-2}\ar[rr]^{} &&  G^{m-1}\ar[rr]^{} && G^{m}\ar[rr]^{} &&G^{m+1}\ar[r]^{} &\cdots\\}$$

		Clearly, $g^i - h^if^i = 0$ for $i\geq m$. Since ${G^{\bullet}}$ is exact in degree $< m$, the sequence $0 \rightarrow
		\Ker(d^{m-2}_{G^{\bullet}}) \rightarrow  G^{m-2} \rightarrow  \Ker(d^{m-1}_{G^{\bullet}}) \rightarrow  0 $ is exact. Note that $d^{m-1}_{G^{\bullet}}(g^{m-1} - h^{m-1}f^{m-1}) = 0$,
		then $g^{m-1} - h^{m-1}f^{m-1}\in  \Hom_R(X^{m-1}, \Ker(d^{m-1}_{G^{\bullet}}) )$. Since $X^{m-1}$ is projective, there exists
		a map $s^{m-1}: X^{m-1} \rightarrow  G^{m-2}$ such that $g^{m-1} - h^{m-1}f^{m-1} = d^{m-2}_{G^{\bullet}} s^{m-1}$. By induction
		we get a homotopy $s : {X^{\bullet}} \rightarrow  {G^{\bullet}}[-1]$ with $s^i = 0$ for any $i\geq m$. Hence $g - hf : X^{\bullet} \rightarrow  G^{\bullet}$ is
		null homotopic, that is,  $g\sim hf$.
	\end{proof}

	\begin{lemma}\label{5.4}  Let $ P^{\bullet}\in \K^{-}(\fP_n)$. If $P^{\bullet}$ is $n$-exact, then $P^{\bullet} = 0$ in $\K(R)$.
	\end{lemma}
	\begin{proof}
		Without loss of generality, we assume that
		$$P^{\bullet}= \cdots \rightarrow  P^{-2}\xrightarrow{d^{-2}} P^{-1}\xrightarrow{d^{-1}}P^{0}\rightarrow  0.$$
		The sequence $0 \rightarrow  \Ker(d^{-1}) \rightarrow  P^{-1} \rightarrow  P^{0} \rightarrow  0$ is exact and $n$-exact, and since
		$P^0$ is $n$-projective, this sequence splits which implies that $\Ker(d^{-1})$ is $n$-projective. Then one can apply the same argument to the sequence $0 \rightarrow  \Ker(d^{-2}) \rightarrow  P^{-2} \rightarrow  \Ker(d^{-1}) \rightarrow  0$ , and so on. Inductively, one has $P^{i- 1}=\Ker(d^{i- 1}) \oplus \Ker(d^{i})$ for
		every $i\leq 0$, and hence $P^{\bullet}$ is the direct sum of contractible complexes of the form
		$\cdots\rightarrow 0\rightarrow  \Ker(d^{i}) \xrightarrow{\Id}  \Ker(d^{i}) \rightarrow 0 \rightarrow  \cdots$. Thus $P^{\bullet} = 0$ in $\K(R)$.
	\end{proof}
	We denote by $\K_{ac}^{b}(\fP_n)$ the homotopy category of bounded exact complexes of $n$-projective modules.
	\begin{lemma}\label{5.5} Assume that $\fP_n$ is closed under kernels of epimorphisms. Let $P^{\bullet}\in \K^{-,\boldsymbol{nb}}(\fP_n)$. If
		$P^{\bullet}$ is exact, then  $P^{\bullet}\in \K_{ac}^{b}(\fP_n)$.
	\end{lemma}
	\begin{proof}
		
		Since $P^{\bullet}\in \K^{-,\boldsymbol{nb}}(\fP_n)$, there exists an integer $m \in\mathbb{Z}$ such that $\Hh^i
		(\Hom_R(Q,{P^{\bullet}})) = 0$ for any $i\leq m$ and any $Q \in\fP_n$. Define
		$$\tau_{\geq m+1}P^{\bullet} := \cdots  \rightarrow  0 \rightarrow  0 \rightarrow  \Im(d^m) \rightarrow  P^{m+1} \rightarrow  P^{m+2} \rightarrow  \cdots  ,$$
		$$\tau_{\leq m}P^{\bullet} := \cdots  \rightarrow  P^{m-2} \rightarrow  P^{m-1} \rightarrow \Ker(d^m) \rightarrow  0\rightarrow 0 \rightarrow  \cdots $$
		Then we obtain an exact sequence of complexes $0 \rightarrow  \tau_{\leq m}P^{\bullet}
		\xrightarrow{f}  P^{\bullet}
		\xrightarrow{g}  \tau_{\geq m+1}P^{\bullet} \rightarrow  0.$ Since
		$\fP_n$ is closed under kernels of epimorphisms by assumption, and $P^{\bullet}$ is exact and upper bounded, it is
		easy to check that $\Im(d^m), \Ker(d^m)\in \fP_n$. Moreover, $\Hh^m(\Hom_R(Q, P^{\bullet})) = 0$ for any $Q\in
		\fP_n$, which implies that the sequence $ 0 \rightarrow  \Ker(d^m) \rightarrow  P^m \rightarrow  \Im(d^m) \rightarrow  0$ is $n$-exact,
		and hence it is split. Thus, the sequence of complexes $0 \rightarrow \tau_{\leq m}P^{\bullet}
		\xrightarrow{f} P^{\bullet}
		\xrightarrow{g}  \tau_{\geq m+1}P^{\bullet}  \rightarrow  0$
		is split degree-wise, and then there is a triangle in the homotopy category $\K(R)$
		$$\tau_{\leq m}P^{\bullet} \xrightarrow{f} P^{\bullet}
		\xrightarrow{g}  \tau_{\geq m+1}P^{\bullet}\xrightarrow{h} \tau_{\leq m}P^{\bullet}[1].$$
		Clearly $h = 0$, and thus $P^{\bullet}= \tau_{\leq m}P^{\bullet}\oplus \tau_{\geq m+1}P^{\bullet}$. Since $P^{\bullet}\in \K^{-,\boldsymbol{nb}}(\fP_n)$, we have
		$\tau_{\leq m}P^{\bullet}\in K^-(\fP_n)$; moreover, $\tau_{\leq m}P^{\bullet}$ is $n$-exact, and then $\tau_{\leq m}P^{\bullet}=0$ by Lemma \ref{5.4}.
		Hence $P^{\bullet}\cong \tau_{\geq m+1}P^{\bullet} \in \K_{ac}^{b}(\fP_n).$
	\end{proof}
	
	\begin{remark}\label{7.5} Note that $\fP_n$ may be not closed under kernels of epimorphisms for $n\geq 1$. Indeed, let $K$ be an algebraic closed field, and  $R$ a  finite-dimensional hereditary $K$-algebra of infinite-representation type. Then there exists an indecomposable $R$-module $M$ with infinite dimension. Consider the exact sequence $0\rightarrow M\rightarrow E_1\rightarrow E_2\rightarrow 0,$ where each $E_i$ is an injective $R$-module. Then each $E_i$ is a direct sum of indecomposable injective modules with with finite dimensions. Recall from \cite[Theorem 1.1]{O82} that pure projective modules over a finite-dimensional $K$-algebra  are exactly any direct sums of indecomposable modules with with finite dimensions. So $E_1$ and $E_2$ are  pure projective modules, but $M$ is not pure projective. It follows by Corollary \ref{fnfn+1} that the class of all pure projective modules over $R$  are exactly the class $\fP_n(R)$  of all $n$-projective modules over $R$ for any $n\geq 1$. Consequently, the class $\fP_n(R)$  need not be closed under all kernels of epimorphisms  for any  $n\geq 1$.
	\end{remark}
	
	Recall from \cite{AF91} that a dg-projective resolution of a complex $X^{\bullet}$ is a quasi-isomorphism
	of complexes $f : P^{\bullet}\rightarrow X^{\bullet} $ with $P^{\bullet}$ dg-projective. By \cite[1.6]{AF91}, every complex has a dg-projective resolution. We denote by $\K^b_{ac}(\fP_n)$ the homotopy category of bounded exact complexes of $n$-projective modules. Now we are in a position to prove the main result of this section.

	\begin{theorem}\label{last}Assume that $\fP_n$ is closed under kernels of epimorphisms. There are  triangulated equivalences for each $n$:
		$$\D^b(R)\cong \D^b_{n}(R)/\K^b_{ac}(\fP_n)\cong \K^{-,\boldsymbol{nb}}(\fP_n)/\K^b_{ac}(\fP_n).$$
	\end{theorem}
	\begin{proof}
		Let $Inc : \K^{-,\boldsymbol{nb}}(\fP_n) \rightarrow  \K^-(R)$ be the embedding functor,
		and $Q : \K^-(R) \rightarrow  \D^-(R)$ be the canonical localization functor. We denote the composition functor by $\eta: \K^{-,\boldsymbol{nb}}(\fP_n) \rightarrow  \D^-(R)$. Since $\eta(\K^b_{ac}(\fP_n)) = 0$, by the universal property of quotient functor we have a unique triangle functor
		$$\overline{\eta}:\K^{-,\boldsymbol{nb}}(\fP_n)/\K^b_{ac}(\fP_n)\rightarrow \D^-(R).$$
		Clearly, $\Im(\overline{\eta})\subseteq \D^b(R)$. By Lemma \ref{5.2}, for any $X^{\bullet}\in \K^{-,b}(\Proj)\cong \D^b(R)$, there
		exists $P^{\bullet} \in \K^{-,\boldsymbol{nb}}(\fP_n)$ such that $X^{\bullet}\cong \overline{\eta}(P^{\bullet})$. Hence the triangle functor $\overline{\eta}$ is dense.
		
		Let $P^{\bullet}_1, P^{\bullet}_2 \in \K^{-,\boldsymbol{nb}}(\fP_n)$ and $\alpha/s \in\Hom_{D^b(R)}(P^{\bullet}_1, P^{\bullet}_2)$, where $s : P^{\bullet}_1\Leftarrow Y^{\bullet}$ is a quasi-isomorphism with $Y^{\bullet}\in \K^-(R)$ and $\alpha : Y^{\bullet} \rightarrow P^{\bullet}_2$ is a morphism in $K^-(R)$. In fact, we can assume $Y^{\bullet}\in \K^{-,b}(R)$ since $P^{\bullet}_1\in \K^{-,\boldsymbol{nb}}(\fP_n)\subseteq \K^{-,b}(R)$. Let $t : X^{\bullet}\rightarrow Y^{\bullet}$ be a dg-projective resolution of $Y^{\bullet}$ , then $t$ is a quasi-isomorphism, and thus we can let
		$X^{\bullet}\in \K^{-,b}(\mathscr{P})$.
		
		For $X^{\bullet} \in  \K^{-,b}(\mathscr{P})$ and $P^{\bullet}_1, P^{\bullet}_2 \in  K^{-,n}(\fP_n)$, there exist integers $m(X^{\bullet}), m(P^{\bullet}_1)$
		and $m(P^{\bullet}_2)$, such that $H^i
		(\Hom_R(M, X^{\bullet})) = 0$ for any $i \leq  m(X^{\bullet})$ and any $M \in  P^{\bullet}$, and
		$\Hh^j (\Hom_R(N, P_k)) = 0$ for any $j \leq  m(P_k)$ and any $N \in  \fP_n (k = 1, 2)$. Let $m =
		\min\{m(X^{\bullet}), m(P^{\bullet}_1), m(P^{\bullet}_2)\}$. By Lemma \ref{5.2}, there is a complex $P^{\bullet} \in  \K^{-,\boldsymbol{nb}}(\fP_n)$ and a
		quasi-isomorphism $f : X^{\bullet} \rightarrow  P^{\bullet}$. Moreover, for morphisms $st : X^{\bullet} \rightarrow  P^{\bullet}_1$ and $\alpha t : X^{\bullet} \rightarrow  P^{\bullet}_2$,
		we have morphisms $g_1 : P^{\bullet} \rightarrow  P^{\bullet}_1$ and $g_2 : P^{\bullet} \rightarrow  P^{\bullet}_2$ satisfying $st \sim g_1f$ and $\alpha t \sim g_2f$ by
		Lemma \ref{5.3}. Then we have the following commutative diagram in $\K^-(R)$:
		$$\xymatrix@R=20pt@C=25pt{ &\ar@{=>}[ld]_{s} Y^{\bullet} \ar[rd]^{\alpha}&\\
			P^{\bullet}_1& \ar@{=>}[l]_{st}   X^{\bullet}\ar@{=>}[u]^{t}\ar@{=>}[d]^{f}\ar[r]^{\alpha t} &P^{\bullet}_2\\
			&  \ar@{=>}[lu]^{g_1} P^{\bullet}\ar[ru]_{g_2} & \\
		}$$
		where the double arrowed morphisms mean quasi-isomorphisms. Note that $g_1$ is also a
		quasi-isomorphism, hence the mapping cone $\Cone^{\bullet}(g_1)$ is exact. Moreover, $\Cone^{\bullet}(g_1)\in
		\K^{-,\boldsymbol{nb}}(\fP_n)$, and hence $\Cone^{\bullet}(g_1)\in \K_{ac}^{b}(\fP_n)$ by Lemma \ref{5.5}. Thus $g_2/g_1$ is a morphism in $\K^{-,\boldsymbol{nb}}(\fP_n)/\K_{ac}^{b}(\fP_n)$ and $\alpha/s = g_2/g_1 =\overline{\eta}(g_2/g_1)$. This implies that the functor $\overline{\eta}$ is full.
		
		It remains to prove $\overline{\eta}$ is faithful. Because the triangle functor $\overline{\eta}$ is full, by \cite[p.446]{R89},
		it suffices to show that it sends non-zero objects to non-zero objects. Suppose
		$P^{\bullet} \in \K^{-,\boldsymbol{nb}}(\fP_n)$ and $\overline{\eta}(P^{\bullet}) = 0$, then $P^{\bullet}$ is exact and it follows from Lemma \ref{5.5} that
		$P^{\bullet} \in \K_{ac}^{b}(\fP_n)$. Hence $\overline{\eta}$ is faithful. This completes the proof.
	\end{proof}



	\vspace{0.5cm}
\textbf{Acknowledgement}.
This work was supported by the NSF of China (No. 12271292), and Youth Innovation Team of Universities of Shandong Province (No. 2022KJ314).

\end{document}